\documentclass[a4paper,12pt,reqno]{amsart}
\usepackage{amsmath}
\usepackage{amsthm}
\usepackage{amsfonts}
\usepackage{amssymb}
\usepackage{amscd}
\usepackage{bm}
\usepackage{color}
\usepackage{draftcopy}
\usepackage{exscale}
\usepackage{extarrows}
\usepackage{epstopdf}
\usepackage{epsfig}
\usepackage{enumerate}
\usepackage{graphics}
\usepackage{graphicx}
\usepackage{latexsym}
\usepackage{mathrsfs}
\usepackage{pifont}
\usepackage{paralist}
\usepackage{esint}
\usepackage{boxedminipage}
\renewcommand\eqref[1]{(\ref{#1})} 
\usepackage[toc,page]{appendix}

\usepackage{hyperref}
\renewcommand\eqref[1]{(\ref{#1})} 

\usepackage[a4paper, hmargin={2.7cm,2.7cm},vmargin={3.3cm,3.3cm}]{geometry}

\newtheorem{theorem}{Theorem}[section]
\newtheorem{proposition}[theorem]{Proposition}
\newtheorem{lemma}[theorem]{Lemma}

\theoremstyle{definition}
\newtheorem{definition}[theorem]{Definition}

\newtheorem{remark}[theorem]{Remark}

\numberwithin{equation}{section}

\newcommand{\supp}{\operatorname{supp}}

\newcommand{\Dom}{\operatorname{Dom}}
\newcommand{\Id}{\operatorname{Id}}

\begin{document}

\title[Lipschitz spaces adapted to Schr\"{o}dinger operators]{Littlewood-Paley and Carleson measure characterizations 
of Lipschitz spaces  \\ adapted to Schr\"{o}dinger operators}

\author[Qing Hong]{Qing Hong${ }^1$}
\thanks{$^1$School of Mathematics and Statistics,
 Jiangxi Normal University,
Nanchang, Jiangxi 330022, China}

\author[Yanfang Xu]{Yanfang Xu${ }^1$}

\author[Guorong Hu]{Guorong Hu${}^{1 \ast}$}
\thanks{$^\ast$Corresponding author, Email: hugr@mail.ustc.edu.cn}






\thanks{The first author is supported by the National Natural Science Foundation of China (Grant No. 12361017) and 
the Natural Science Foundation of Jiangxi Province (Grant No. 20242BAB25002).
The third author is supported by the National Natural Science Foundation of China (Grant No. 12461018)}

\begin{abstract}
Let $L =-\Delta +V$ be a Schr\"{o}dinger operator on $\mathbb{R}^n$, $n \geq 3$, with the potential $V$ being nonnegative and 
belonging to the reverse H\"{o}lder class $RH_q$ for some $q >n/2$. 
For $0< \alpha <2$, the Lipschitz space $\Lambda_L^\alpha(\mathbb{R}^n)$
adapted to $L$ is defined as the space of all measurable functions $f$ on $\mathbb{R}^n$ such that
\[
\|f\|_{\Lambda_L^\alpha} := \|\rho(\cdot)^{-\alpha}f(\cdot)\|_{L^\infty}+ \sup_{h \in \mathbb{R}^n \backslash \{0\}}
  \frac{\|f(\cdot + h) + f(\cdot -h) -2 f(\cdot)\|_{L^\infty}}{|h|^\alpha} <\infty,
\]
where $\rho$ is the critical radius function related to $L$.
In this paper, we establish characterizations of $\Lambda^\alpha_L(\mathbb{R}^n)$ in terms of Littlewood-Paley-type decompositions and generalized Carleson measures, for $0< \alpha < 2 -(n /q)$.

\end{abstract}

\maketitle

\section{Introduction and statement of main results}

\allowdisplaybreaks

\subsection{Motivation} 
The classical Lipschitz spaces (also called H\"{o}lder-Zygmund spaces) have several equivalent definitions, including via differences, the Poisson semigroup, and the heat semigroup; see, e.g., \cite[Chapter V]{Stein-singular}, \cite[Section 1.2.2]{Triebel-2} and \cite[Section 6.3]{Grafakos-Modern}. We take the difference definition as the original definition of Lipschitz spaces.  Following the notation of \cite[Section~6.3]{Grafakos-Modern}, we define, for $h \in \mathbb{R}^n$, the difference
operator $D_h$ by  
\begin{align*}
 D_h (f)(x) =f(x+h) -f(x), \quad x \in \mathbb{R}^n,  
\end{align*}
for any  function $f : \mathbb{R}^n \rightarrow \mathbb{C}$. Then for $k \in \mathbb{N}$,  the
$k$-th difference operator $D_h^k$ is defined iteratively via
\[
D^1_h: = D_h \quad \text{and} \quad D^k_h : =D_h(D_h^{k-1}), \quad k \geq 2. 
\]
The classical inhomogeneous and homogeneous Lipschitz spaces are defined  as follows; see 
\cite[Definitions 6.3.3 and 6.3.4]{Grafakos-Modern}.
\begin{definition}  \label{def:hom inhom}
\begin{enumerate}[(i)]
    \item For $\alpha >0$, the inhomogeneous Lipschitz space ${\Lambda}^\alpha (\mathbb{R}^n)$
is defined as the set of all measurable functions $f$ such that 
\begin{align*}
 \|f\|_{\Lambda^\alpha} : =\|f\|_{L^\infty} + \sup_{h \in \mathbb{R}^n \backslash \{0\} }\frac{\|D_h^{\lfloor \alpha \rfloor +1} (f)\|_{L^\infty }}{|h|^\alpha}   <\infty.
\end{align*}
\item For $\alpha >0$, the   homogeneous Lipschitz space $\dot{\Lambda}^\alpha (\mathbb{R}^n)$
is defined as the set of all measurable functions $f$ such that 
\begin{align*}
 \|f\|_{\dot{\Lambda}^\alpha} : = \sup_{h \in \mathbb{R}^n \backslash \{0\} }\frac{\|D_h^{\lfloor \alpha \rfloor +1}  (f)\|_{L^\infty}}{|h|^\alpha}   <\infty.
\end{align*}
\end{enumerate}
\end{definition}

\begin{remark} \label{rmk:remarks on classical}
\begin{enumerate}[(a)] 
\item Obviously, for every $\alpha >0$, we have  $\Lambda^\alpha (\mathbb{R}^n) \subset \dot{\Lambda}^\alpha (\mathbb{R}^n)$.
   \item  It is pointed out in \cite[p. 72]{TW} that if 
    $f \in \dot{\Lambda}^\alpha(\mathbb{R}^n)$ then $f$ may be modified
   on a set of measure zero so that it becomes continuous. Therefore, in many references (e.g., \cite[Definitions 6.3.3 and 6.3.4]{Grafakos-Modern}), functions in Lipschitz spaces are directly assumed to be continuous.
   \item  Elements of $\dot{\Lambda}^\alpha (\mathbb{R}^n)$ have at most polynomial growth
at infinity and thus they are tempered distributions; see \cite[p. 26]{Grafakos-Modern}.
   \item  For any polynomial $P$ of degree  $ \leq \lfloor \alpha \rfloor$, we have $D_h^{\lfloor \alpha\rfloor +1} P =0$. This means that
   $\|\cdot\|_{\dot{\Lambda}^\alpha (\mathbb{R}^n)}$ is not a norm but only a seminorm. 
   To make it a norm, we need to consider the elements of $\dot{\Lambda}^\alpha (\mathbb{R}^n)$ as functions modulo
   polynomials of degree  $\leq \lfloor \alpha \rfloor$.
   \item If we replace $D_h^{\lfloor \alpha\rfloor +1}$ with $D_h^m$ for any integer $m > \lfloor \alpha\rfloor +1$ in (i) of Definition~\ref{def:hom inhom}, then we will get the same 
   inhomogeneous Lipschitz space. For example, when $0< \alpha <1$,
   the  spaces $\Lambda^\alpha(\mathbb{R}^n)$ defined via first difference and second difference are equivalent.
   See \cite[Proposition 8 in Chapter V]{Stein-singular} and \cite[Theorem 1.2.2 (i)]{Triebel-2}.
\item The Lipschitz spaces  can be characterized in terms of Littlewood-Paley-type decompositions and (generalized) Carleson measures; see Section \ref{sect:preliminaries} for some details. 

\end{enumerate}
\end{remark}

We now turn to Lipschitz spaces adapted to Sch\"{o}dinger operators. 
Let 
\[
L =-\Delta + V
\]
be a Schr\"{o}dinger operator on $\mathbb{R}^n$, $n \geq 3$. Throughout this paper, 
we assume that the potential $V$ is nonnegative and belongs to the reverse H\"{o}lder class $RH_q$ for some $q > n/2$, that is, there exists a constant $C >0$
such that 
\begin{align*}
\left( \frac{1}{|B|} \int_B V(x)^q dx\right)^{1/q} \leq C \frac{1}{|B|} \int_B V(x)dx
\end{align*}
holds  for every ball $B$ in $\mathbb{R}^n$.
Following \cite{ShenAIF}, we define the critical radius function $\rho(\cdot)$ related to $L$ by
\begin{align*}
\rho(x): =\sup \left\{r >0: \frac{r^2}{|B(x,r)|}\int_{B(x,r)}V(y)dy \leq 1\right\}, \quad x \in \mathbb{R}^n.
\end{align*}
The Lipschitz spaces $\Lambda^\alpha_L(\mathbb{R}^n)$ adapted to 
$L$ were first introduced by Bongioanni, Harboure and Salinas \cite{BHS-weighted} in the case $0< \alpha <1$ via first difference,
and then extended by De Le\'{o}n-Contreras and Torrea \cite{DT} to the case $0< \alpha <2$ via second difference.  

\begin{definition} \label{def:Lipschitz} {\rm (cf. \cite[Definition 1.1]{DT})}
For $0 < \alpha <2$, the Lipschitz space $\Lambda^\alpha_{L} (\mathbb{R}^n)$ adapted to $L$ is defined as 
the space of all measurable functions $f$ on $\mathbb{R}^n$ such that
\begin{align*}
 \|f\|_{\Lambda^\alpha_{L}}:= \|\rho(\cdot)^{-\alpha}f(\cdot)\|_{L^\infty} +  \sup_{h \in \mathbb{R}^n \backslash \{0\}}
  \frac{\|D_h^2 (f)\|_{L^\infty}}{|h|^\alpha} <\infty.
\end{align*}
\end{definition}

\begin{remark} 
\begin{enumerate}[(a)]
    \item When $V \equiv 0$, 
we have $\rho \equiv +\infty$ and hence $\|\rho(\cdot)^{-\alpha}f(\cdot) \|_{L^\infty} =0$, 
so the Lipschitz space $\Lambda_{-\Delta}^\alpha (\mathbb{R}^n)$ ($0< \alpha <2$) adapted to
the Laplacian $-\Delta$ coincides with
the classical homogeneous  Lipschitz space $\dot{\Lambda}^\alpha (\mathbb{R}^n)$.
\item When $V \equiv 1$, we have $\rho \equiv 1$
and hence $\|\rho(\cdot)^{-\alpha}f(\cdot)\|_{L^\infty} = \|f\|_{L^\infty}$,  so 
the Lipschitz space $\Lambda_{-\Delta+\Id}^\alpha (\mathbb{R}^n)$ ($0< \alpha <2$) adapted to
  $-\Delta + \Id$ (where $\Id$ is the identity operator) coincides with 
the classical inhomogeneous Lipschitz  space $\Lambda^\alpha (\mathbb{R}^n)$.
\end{enumerate}
\end{remark}
 
The Lipschitz spaces $\Lambda_L^\alpha (\mathbb{R}^n)$ adapted to $L$ have been analyzed by several authors.
Bongioanni  et al. \cite{BHS-weighted}  first introduced the space $\Lambda_L^\alpha(\mathbb{R}^n)$, for $0< \alpha <1$, via the norm
\[
 \|f\|_{\Lambda^\alpha_L}:= \|\rho(\cdot)^{-\alpha}f(\cdot)\|_{L^\infty} +  \sup_{h \in \mathbb{R}^n \backslash \{0\}}
  \frac{\|D_h(f)\|_{L^\infty}}{|h|^\alpha},
\]
and  showed that this space coincides with the space $BMO_L^\alpha(\mathbb{R}^n)$ which  is 
 defined as the space of locally integrable functions such that
 \[
 \int_B \left|f(x)-\frac{1}{|B|}\int_B f(y)dy\right|dx \leq C |B|^{1 +\frac{\alpha}{n}} \quad \text{for every ball } B =B(x_B,r_B),
 \]
 and
 \[
 \int_B |f(x)|dx \leq C |B|^{1 +\frac{\alpha}{n}} \quad \text{if } r_B \geq \rho(x_B).
 \]
They also study the behavior of the fractional integral $L^{-\beta/2}$ on   $\Lambda_L^\alpha(\mathbb{R}^n)$ for $0< \alpha <1$, $\beta \geq 0$
and $\alpha + \beta < \min\{1,2 -(n/q)\}$.
 Ma {et al.} \cite{MSTZ} prove a characterization of the space $\Lambda_L^\alpha(\mathbb{R}^n)$, for $0<\alpha <1$, in  terms of the growth of 
fractional derivatives of  the Poisson semigroup $e^{-t\sqrt{L}}$ and Carleson measures, and studied the behavior of the
 fractional powers $L^{\beta/2}$ (positive and negative) and  multipliers of Laplace transform type on the space $\Lambda_L^\alpha(\mathbb{R}^n)$.
De Le\'{o}n-Contreras and Torrea \cite{DT} recently extended the 
definition of $\Lambda_L^\alpha(\mathbb{R}^n)$ to $0< \alpha <2$ via second difference (see Definition~\ref{def:Lipschitz}), and 
generalized the regularity results obtained in \cite{MSTZ}.
In the case of the Hermite operator $H =-\Delta + |x|^2$, the associated H\"{o}lder spaces $C^{k,\alpha}_H(\mathbb{R}^n)$, for $k \in \mathbb{N}_0$ and $0< \alpha \leq 1$,  
 were investigated by Stinga and Torrea \cite{ST1}.

Motivated by the Littlewood-Paley and Carleson measure characterizations of the classical Lipschitz spaces (see Proposition \ref{prop:classical L-P homogeneous A} and Proposition \ref{prop: Littlewood-Paley homogeneous B} in Section \ref{sect:preliminaries}), and the Carleson measure characterization 
of $\Lambda^\alpha_L (\mathbb{R}^n)$ for $0< \alpha <1$ due to Ma et al. \cite{MSTZ}, 
the purpose of the present paper is to establish Littlewood-Paley and Carleson measure characterizations for
$\Lambda^\alpha_L(\mathbb{R}^n)$ in the range $0< \alpha < (2-n/q)$.

\subsection{Main results}
To state our results explicitly, we need to introduce some notions and terminologies.
First we give the definition of the test functions  adapted to the operator ${L}$, which was first introduced by 
Kerkyacharian and Petrushev \cite{KP}.
\begin{definition}(cf. \cite[Section 5.2]{KP})
A function $\phi \in L^2 (\mathbb{R}^n)$ is said to be a \textit{test function adapted to ${L}$}, 
if $\phi \in \bigcap_{m \in \mathbb{N}}\Dom({L}^m)$ and for each $m,\ell \in \mathbb{N}_0$, 
\begin{align} \label{eq:seminorm}
\|\phi\|_{(m,\ell)}:= \sup_{x \in \mathbb{R}^n} (1 +|x|)^{\ell}|{L}^m \phi (x)|  < \infty.
\end{align}
We denote by $\mathcal{S}_L(\mathbb{R}^n)$ the set of all test functions adapted to ${L}$.
\end{definition}
Endowed with the family $\{\|\cdot\|_{(m,\ell)}\}_{m,\ell \in \mathbb{N}_0}$ of seminorms, $\mathcal{S}_{L}(\mathbb{R}^n)$ is a Frech\'{e}t space; 
see \cite[Proposition 5.3]{KP}. 
Elements in the topological dual of $\mathcal{S}_{L}(\mathbb{R}^n)$, denoted by $\mathcal{S}_{L}'(\mathbb{R}^n)$, 
are naturally considered as
\textit{distributions adapted to ${L}$}.
The duality pairing between $f \in \mathcal{S}_{L}'(\mathbb{R}^n)$ and $g \in \mathcal{S}_{L}(\mathbb{R}^n)$ will be denoted by $( f, g )$. 

\medskip

Note that when $L=-\Delta$, the space $\mathcal{S}_{L}(\mathbb{R}^n)$ (resp. $\mathcal{S}_{L}'(\mathbb{R}^n)$) coincides
with the classical Schwartz function space $\mathcal{S}(\mathbb{R}^n)$ (resp. the classical tempered distribution space $\mathcal{S}'(\mathbb{R}^n)$);
see \cite[Proposition 5.6]{KP}.

\medskip
For $f \in \mathcal{S}_{L}'(\mathbb{R}^n)$ and $m \in \mathbb{N}_0$, we define 
${L}^m f$ as an element in $\mathcal{S}_L'(\mathbb{R}^n)$ such that 
\[
({L}^mf, \phi) = (f, {L}^m\phi) \quad \text{for all } \phi \in \mathcal{S}_{L}(\mathbb{R}^n).
\]

We also need the notion of generalized polynomials adapted to ${L}$ which are defined as follows. 

\begin{definition}(see \cite[Section 3.2]{GKKP1}) A distribution $P \in \mathcal{S}_L'(\mathbb{R}^n)$ is said to be a \textit{generalized polynomial adapted to} $L$
if there exists $m \in \mathbb{N}_0$ such that ${L}^m P =0$ in $\mathcal{S}_L'(\mathbb{R}^n)$.  The space of all generalized polynomials adapted to $L$ 
is denoted by $\mathcal{P}_{L}(\mathbb{R}^n)$. 
\end{definition}

\begin{remark}
When $L =-\Delta$, $\mathcal{P}_L(\mathbb{R}^n)$ coincides with the space
of all classical polynomials on $\mathbb{R}^n$. This fact was proved in \cite[Section 3.3]{GKKP1}.
\end{remark}

Since $V$ is nonnegative and locally integrable on $\mathbb{R}^n$, 
the Schr\"{o}dinger operator $L = -\Delta +V$, originally defined on $C_0^\infty (\mathbb{R}^n)$, 
extends to a densely defined nonnegative self-adjoint operator
on $L^2(\mathbb{R}^n)$, which is also denoted by $L$. 

Denote by $\{E(\lambda):\lambda \geq 0\}$ the spectral resolution of $L$. For any Borel measurable function $\Phi : [0,\infty) \rightarrow \mathbb{C}$, the ``multiplier operator'' $\Phi(\sqrt{L})$
can be defined by the spectral theorem according to the prescription 
\[
\Phi(\sqrt{L}) = \int_0^\infty \Phi(\sqrt{\lambda}) dE(\lambda).
\]

To recall some results concerning smooth functional calculus related to $L$, we introduce
the class $\mathcal{A}([0,\infty))$ as follows:
\begin{equation} \label{eq:def of A}
\begin{split}
\mathcal{A}([0,\infty)) := \big\{\Phi \in C^\infty ([0,\infty)): & \; \Phi \text{ can be extended }\\
& \text{ to an even Schwartz function on } (-\infty, \infty)\big\}.   
\end{split}
\end{equation}
It is clear that if $\Phi \in C^\infty ([0,\infty))$ and $\supp \Phi \subset [a,b]$ for some $0< a< b <\infty$, 
then $\Phi \in \mathcal{A}([0,\infty))$.
An interesting result  in \cite[Proposition 5.3]{KP} is that if $\Phi\in \mathcal{A}([0,\infty))$,
then the operator $\Phi(\sqrt{L})$ is an integral operator with a kernel $K_{\Phi(\sqrt{L})}(\cdot,\cdot)$ such that
\begin{equation} \label{eq:Schwartz with x fixed}
\begin{split}
K_{\Phi(\sqrt{L})}(x,\cdot) \in \mathcal{S}_L(\mathbb{R}^n) \ & \text{with } x \text{ fixed } \\
&\text {and }  K_{\Phi(\sqrt{L})}(\cdot, y) \in \mathcal{S}_L(\mathbb{R}^n)  \ \text{with } y \text{ fixed}. 
\end{split}
\end{equation}
This result holds for all (abstract) nonnegative self-adjoint operator whose heat kernel 
obeys a pointwise Gaussian estimate, and in particular, this holds for Schr\"{o}dinger operators with nonnegative locally integrable potentials
(see \eqref{eq:GUE} in Section \ref{sect:kernel estimates}).

Assuming that $\Phi\in\mathcal{A}([0,\infty))$, we continue discussing 
properties of the operator $\Phi(\sqrt{L})$. 
From \cite[Proposition 5.3]{KP} we see that  $\Phi(\sqrt{L})$ maps continuously $\mathcal{S}_{L}(\mathbb{R}^n)$ into 
$\mathcal{S}_{L}(\mathbb{R}^n)$. 
Hence $\Phi(\sqrt{L})$ extends to be a continuous operator
$\Phi(\sqrt{{L}}): \mathcal{S}_{L}'(\mathbb{R}^n)\rightarrow \mathcal{S}_{L}'(\mathbb{R}^n)$
in the following sense: for any $f \in \mathcal{S}_{L}'(\mathbb{R}^n)$,
\begin{align} \label{eq:Phi L f 1}
\big(\Phi(\sqrt{{L}})f, \phi \big) := \big(f, \Phi(\sqrt{L})\phi \big), \quad \phi \in \mathcal{S}_{L}(\mathbb{R}^n).   
\end{align}
It is easy to see that (see \cite[(3.14)]{GKKP1}), if $\Psi, \Gamma  \in \mathcal{A}([0,\infty))$ and
$\Omega: = \Psi   \Gamma$, then for every $f \in \mathcal{S}_L'(\mathbb{R}^n)$, we have
\begin{align} \label{eq:commute}
\Psi(\sqrt{L})\Gamma(\sqrt{L})f   =\Gamma(\sqrt{L})\Psi(\sqrt{L})f = \Omega(\sqrt{L})f
\quad \text{in } \mathcal{S}_L'(\mathbb{R}^n).
\end{align}
We also note that (cf. \cite[Proposition 3.2]{GKKP1}) the distribution $\Phi(\sqrt{{L}})f$ defined by 
\ref{eq:Phi L f 1} is identified with the pointwise defined function  
\begin{align} \label{eq:Phi L f}
\Phi(\sqrt{L})f (x) = \big(f, K_{\Phi(\sqrt{L})}(x,\cdot) \big), \quad x \in \mathbb{R}^n.
\end{align}

\begin{remark}
Note that if $f \in \Lambda^\alpha_L (\mathbb{R}^n)$, then $f$ can be regarded as a distribution
in $\mathcal{S}_L'(\mathbb{R}^n)$. Indeed, from Definitions \ref{def:hom inhom}  and 
\ref{def:Lipschitz} we see that $\Lambda_L^\alpha (\mathbb{R}^n)\subset \dot{\Lambda}^\alpha(\mathbb{R}^n)$, and hence any  $f \in \Lambda^\alpha_L (\mathbb{R}^n)$
have at most polynomial growth at infinity. See Remark \ref{rmk:remarks on classical} (c).
 \end{remark}

Our first main result  is the following
\begin{theorem} \label{thm:main-1}
Assume that $|\{x\in \mathbb{R}^n: V(x) >0\}| \neq 0$.
Let $0< \alpha < 2 -(n /q)$ and let $\Phi \in C^\infty ([0,\infty))$  satisfy
\begin{align} \label{eq:Phi condition}
\supp \Phi \subset  [1/2,2] \quad \text{and} \quad |\Phi(\lambda)| \geq c >0
\ \text{  for  } \  3/5  \leq \lambda \leq 5/3.
\end{align}
 \begin{enumerate}[\rm (i)]
     \item 
     There is a positive constant $C$ such that for every $f \in \Lambda_{L}^\alpha (\mathbb{R}^n)$,  
    \[
     \sup_{j \in \mathbb{Z}} 2^{j\alpha}\|\Phi(2^{-j}\sqrt{{L}})f\|_{L^\infty} \leq C \|f\|_{\Lambda_{L}^\alpha}.
     \]
     \item For a distribution $f \in \mathcal{S}_{L}'(\mathbb{R}^n)$, if 
     \begin{align} \label{eq:finite Besov norm}
     \sup_{j \in \mathbb{Z}}2^{j\alpha} \|\Phi(2^{-j}\sqrt{{L}})f\|_{L^\infty}<\infty,
     \end{align}
     then there exists a generalized polynomial $P_f \in \mathcal{P}_L$ such that 
     $f -P_f \in \Lambda_{L}^\alpha (\mathbb{R}^n)$. Moreover, there is a constant $C'$ such that 
     for every $f \in \mathcal{S}_L'(\mathbb{R}^n)$ satisfying \eqref{eq:finite Besov norm},
     \[
     \|f-P_f\|_{\Lambda_{L}^\alpha} \leq C' \sup_{j \in \mathbb{Z}}2^{j\alpha} \|\Phi(2^{-j}\sqrt{{L}})f\|_{L^\infty}.
     \]     
     \end{enumerate}
\end{theorem}

\begin{remark}
\begin{enumerate}[(a)]
    \item Theorem \ref{thm:main-1} generalizes the classical result involving the (positive) Laplacian $-\Delta$
(see Proposition~\ref{prop: Littlewood-Paley homogeneous B} in Section \ref{sect:preliminaries})  
to the setting of Schr\"{o}dinger operators $L =-\Delta + V$.
\item The earliest  characterization of Lipschitz spaces using Littlewood-Paley type operators
(building from the Poisson kernel)
apperas in the work of Hardy and Littlewood \cite{Hardy-Littlewood}.
These and other characterizations were obtained and extensively studied in higher dimensions by Taibleson
Taibleson \cite{Taibleson-1, Taibleson-2, Taibleson-3}; see also Stein \cite[Chapter~V]{Stein-singular}. The connection
of Lipchitz spaces with modern Littlewood-Paley decomposition (using kernels in $\mathcal{S}(\mathbb{R}^n)$) 
was brought to the surface by Peetre  \cite{Peetre}. Lipschitz spaces on stratified groups were studied by Folland \cite{Folland} and Krantz \cite{Krantz}, and 
a characterization of Lipschitz spaces on stratified groups in terms of Littlewood-Paley decompositions
was proved by Hu \cite{Hu-CMJ}. Littlewood-Paley characterizations of some multi-parameter Lipschitz spaces 
were established in \cite{HH, HHLT-1, HHLT-2}.
\end{enumerate}
\end{remark}

Now we turn to the Carleson measure characterization of $\Lambda_L^\alpha (\mathbb{R}^n)$. We 
will show that  $\Lambda^\alpha_L (\mathbb{R}^n)$ can be characterized via the $\alpha$-Carleson
measures which are defined as follows.

\begin{definition} \label{def:alpha Carleson continuous}
Let $\alpha \geq 0$. The {\it $\alpha$-Carleson measure space} $\mathcal{C}^\alpha$ is defined as the set
of all nonnegative Borel measures $\mu$ on $\mathbb{R}_+^{n+1} :=\mathbb{R}^n \times (0,\infty)$ such that
\begin{align} \label{eq:def of carleson}
\|\mu\|_{\mathcal{C}^\alpha} := \sup_{B} \frac{\mu(\widehat{B})}{|B|^{1 +\frac{2\alpha}{n}}} <\infty,
\end{align}
where the supremum is over all balls $B \subset \mathbb{R}^n$, and 
$\widehat{B} := \{(x,t): x \in B, \ 0< t <r_B\}$, with $r_B$ being the radius of $B$.
\end{definition}

Note that when $\alpha =0$, the $\alpha$-Carleson measure space $\mathcal{C}^\alpha$ coincides with the classical Carleson measure space
$\mathcal{C}$ (cf. Definition \ref{def:classical Carleson}).

For each $\gamma >0$, we denote
\[
\mathcal{F}_\gamma(\mathbb{R}^n) := \left\{f: f \text{ is measurable on } \mathbb{R}^n \text{ and } 
\int_{\mathbb{R}^n} \frac{|f(x)|}{(1 + |x|)^{n +\gamma}} dx <\infty\right\}.
\]
We then set 
\begin{align*}
\mathcal{F}(\mathbb{R}^n) := \bigcup_{\gamma >0} \mathcal{F}_\gamma (\mathbb{R}^n).
\end{align*}
It is easy to see that $\mathcal{F}(\mathbb{R}^n) \subset \mathcal{S}_L'(\mathbb{R}^n)$.

\medskip

For each $k \in \mathbb{N}_0$, let $\Theta_{(k)}(\cdot)$ be the function on $[0,\infty)$  defined by
\begin{align} \label{eq:def of Theta}
\Theta_{(k)} (\lambda):=  \lambda^{2k}e^{-\lambda^2}, \quad \lambda \in [0,\infty).
\end{align}
It is clear that $\Theta_{(k)}(\cdot) \in \mathcal{A}([0,\infty))$. Hence, for any $t>0$ and any $f \in \mathcal{F}(\mathbb{R}^n) \subset  \mathcal{S}_L'(\mathbb{R}^n)$,
$\Theta_{(k)}(t \sqrt{L}) f$ a well-defined distribution according to \eqref{eq:Phi L f 1}, and this 
distribution coincides with a pointwise defined function in view of \eqref{eq:Phi L f}.
In what follows, for each $k \in \mathbb{N}$ and $f \in \mathcal{F}(\mathbb{R}^n)$, we defined 
the associated measure $d\mu_{(k,f)}$ on $\mathbb{R}_+^{n+1}:=\{(x,t): x \in \mathbb{R}^n, t >0\}$ by setting 
\begin{align} \label{eq:def of mukf}
d\mu_{(k,f)}:= \big|\Theta_{(k)}(t\sqrt{L})f(x)\big|^2 \frac{dxdt}{t}= \big|(t^2L)^k e^{-t^2L}f(x)\big|^2\frac{dx dt}{t}.
\end{align}
Then, according to \eqref{eq:def of carleson},
\begin{equation} \label{eq:def of d mu k f}
\begin{split}
\|d\mu_{(k,f)}\|_{\mathcal{C}^\alpha}& = \sup_{\substack{B\subset \mathbb{R}^n \atop 
B: \text{ ball}}} \frac{1}{|B|^{1 +\frac{2\alpha}{n}}} \int_0^{r_B} \int_{B}\big|\Theta_{(k)}(t\sqrt{L})f(x)\big|^2\frac{dx dt}{t}.
\end{split}
\end{equation}

 With the above notation, our second main result can be stated as follows.

\begin{theorem} \label{thm:main-2}
Assume that $|\{x\in \mathbb{R}^n: V(x) > 0\}| \neq 0$. Let $0< \alpha < 2-(n/q)$ and $k \in \mathbb{N}$ 
with $k > {\alpha}/{2}$. Then we have:
\begin{enumerate}[\rm (i)]
    \item If $f \in \Lambda_L^\alpha(\mathbb{R}^n)$, then $f \in \mathcal{F}(\mathbb{R}^n)$ and $d\mu_{(k,f)} \in \mathcal{C}^\alpha$. Moreover,  
    there is a constant $C$ such that for every $f \in \Lambda_L^\alpha (\mathbb{R}^n)$,
    \begin{align} \label{eq:1.7-i}
         \big\|d\mu_{(k,f)}\big\|_{\mathcal{C}^\alpha} \leq C \|f\|_{\Lambda_L^\alpha}^2.
    \end{align}
    \item If $f \in \mathcal{F}(\mathbb{R}^n)$ such that $d\mu_{(k,f)} \in \mathcal{C}^\alpha$, then $f \in \Lambda_L^\alpha (\mathbb{R}^n)$. Moreover, there is a constant $C'$ such that for every $f \in \mathcal{F}(\mathbb{R}^n)$,
    \begin{align*}
 \|f\|_{\Lambda_L^\alpha}^2 \leq C'\big\|d\mu_{(k,f)}\big\|_{\mathcal{C}^\alpha}.
\end{align*}
\end{enumerate}
\end{theorem}

\begin{remark}
\begin{enumerate}[(a)]
    \item Theorem \ref{thm:main-2} generalizes the Carleson measure characterization
    of the classcial Lipschitz spaces via the Laplacian $-\Delta$ (see Proposition \ref{prop:Carleson measure classical} in Section~\ref{sect:preliminaries}) to the setting of Schr\"{o}dinger operators $L =-\Delta +V$.
    However, unlike Proposition~\ref{prop:Carleson measure classical}, the $\alpha$-Carleson measure in 
    Theorem \ref{thm:main-2} is a continuous version; moreover, since we assume that $f$ satisfies the heat size condition for $L$, we do not need to consider the equivalence in the sense of modulo (generalized) polynomials.  
    \item As noted earlier, 
Ma  et al. \cite{MSTZ} obtained a characterization of the space $\Lambda_L^\alpha(\mathbb{R}^n)$, for $0<\alpha <1$, in
terms of Carleson measures and fractional derivatives of any order of the Poisson semigroup; see \cite[Theorem 1.3]{MSTZ}. Their approach is 
 based on the following two facts: 
 \begin{itemize}
 \item for $0< \alpha <1$, $\Lambda^\alpha_L(\mathbb{R}^n) = BMO_L^\alpha(\mathbb{R}^n)$  (cf.  \cite[Proposition 4]{BHS-weighted}); 
 \item for $0< \alpha <1$, $BMO_L^\alpha(\mathbb{R}^n)$ is
 the dual of the atomic Hardy space $H_{{\rm at},L}^{\frac{\alpha}{n +\alpha}}(\mathbb{R}^n)$
 associated with $L$ (cf. \cite[Theorem 4.5]{MSTZ}).  
 \end{itemize}
 However,  for $\alpha \geq 1$, 
 the $H^{\frac{n}{n+\alpha}}_{{\rm at},L}(\mathbb{R}^n)$--$ \Lambda_L^\alpha(\mathbb{R}^n)$ duality fails unless the definition of $L$-atoms is refined (see the remark at the bottom of page 826 of \cite{MSTZ}).  
Therefore, to prove the Carleson measure characterization of $\Lambda^\alpha_L (\mathbb{R}^n)$ for $0< \alpha < 2 -(n /q)$, we
 will use a different strategy, which relies on Calder\'{o}n reproducing formulae adapted to the operator $L$ proved 
 in \cite{KP, GKKP1}, the heat semigroup characterization of $\Lambda_L^\alpha (\mathbb{R}^n)$ established in \cite{DT},
 and suitable almost orthogonality estimates.
\end{enumerate}
\end{remark}

The rest of this paper is organized as follows. In Section \ref{sect:preliminaries}, we present some preliminaries, including the Littlewood-Paley and Carleson measure characterizations of the classical Lipschitz spaces, 
some properties of the critical radius function $\rho$, the inhomogeneous and homogeneous Calder\'{o}n reproducing formulae adapted to the operator $L$,
and heat semigroup characterization of $\Lambda_L^\alpha (\mathbb{R}^n)$. In Section \ref{sect:kernel estimates}, we 
present some useful kernel estimates, which play an important role in the proofs of main results. Section \ref{sect:proof of first theorem} and
Section \ref{sect:proof of second theorem} are devoted to the proofs of Theorem \ref{thm:main-1} and Theorem \ref{thm:main-2}, respectively.

\medskip
\noindent
{\it Notation.} The set of all positive integers is denoted by $\mathbb{N}$,
while the set of all nonnegative integers is denoted by $\mathbb{N}_0$. If $a \in (0,\infty)$, we denote 
$\lfloor a \rfloor : =\max\{k \in \mathbb{N}_0: k \leq a\}$. For any $a,b \in \mathbb{R}$, we denote 
$a \vee b:= \max\{a,b\}$ and $a \wedge b := \min\{a,b\}$.
The Lebesgue measure of a measurable set $E \subset \mathbb{R}^n$ is denoted by $|E|$.
Throughout, we will use $c,c',C,C'$ to denote positive constants, which are independent of the main
variables involved and whose values may vary at every occurrence.
By writing $f \lesssim g$, we mean $f \leq Cg$.
The notation $f \asymp g$ will stand for $C \leq f/g \leq C'$.

\section{Preliminaries} \label{sect:preliminaries}

\subsection{Littlewood-Paley characterizaiton of the classical Lipschitz spaces}

The Lipschitz spaces $\Lambda^\alpha(\mathbb{R}^n)$ and
$\dot{\Lambda}^\alpha (\mathbb{R}^n)$ can be characterized by Littlewood-Paley decompositions. 
To give the precise formulation of such a characterization, we will use the following notation.   
For a function $\phi \in \mathcal{S}(\mathbb{R}^n)$, its Fourier transform $\mathcal{F}\phi$ 
and inverse Fourier transform $\mathcal{F}^{-1}\phi$ are defined by
\begin{align}  \label{eq:def of Fourier}
\mathcal{F}\phi (\xi) := \frac{1}{(2\pi)^{\frac{n}{2}}} \int_{\mathbb{R}^n} \phi (\xi)e^{-i x\cdot \xi} d\xi
\quad \text{and} \quad \mathcal{F}^{-1}\phi (\xi): = \frac{1}{(2\pi)^{\frac{n}{2}}} \int_{\mathbb{R}^n} \phi (x)e^{i x\cdot \xi} dx,
\end{align}
respectively. It is well known that $\mathcal{F}$ and $\mathcal{F}^{-1}$ are continuous on $\mathcal{S}(\mathbb{R}^n)$, and hence they are continuous on $\mathcal{S}'(\mathbb{R}^n)$ by duality.   
If $\phi \in \mathcal{S}(\mathbb{R})$ and $f \in \mathcal{S}'(\mathbb{R}^n)$, we write
\begin{align} \label{eq:notation D}
 \phi(D)f: = \mathcal{F}^{-1}(\phi \cdot \mathcal{F} f).   
\end{align}
Then the Littlewood-Paley characterization of the homogeneous Lipschitz spaces $\dot{\Lambda}^\alpha (\mathbb{R}^n)$
can be stated as follows; see \cite[Theorems~6.3.6]{Grafakos-Modern}.

\begin{proposition}  \label{prop:classical L-P homogeneous A} 
Let $\alpha >0$ and let $\phi \in \mathcal{S}(\mathbb{R}^n)$ such that  
\begin{align*} 
\supp {\phi} \subset \{\xi \in \mathbb{R}^n: 1/2  \leq |\xi | \leq 2\} \ \  \mbox{and} \ \  |{\phi}  (\xi)| \geq c >0 \  \text{for} \ 3/5  \leq |\xi| \leq 5/3 .
\end{align*}
\begin{enumerate}[\rm (i)]
    \item There is a constant $C =C(n,\alpha)$ such that for every $f \in \dot{\Lambda}^\alpha (\mathbb{R}^n)$,  
     \begin{align*}
      \sup_{j \in \mathbb{Z}} 2^{j\alpha}\|\phi(2^{-j}D)f\|_{L^\infty}
      \leq C \|f\|_{\dot{\Lambda}^\alpha}.
     \end{align*}
\item  For every tempered distribution $f$ that satisfies
    \begin{align} \label{eq:sup j finite 23}
       \sup_{j \in \mathbb{Z}}2^{j\alpha} \|\phi(2^{-j}D)f\|_{L^\infty}
      <\infty,
      \end{align}
     there exists a polynomial $P$ such that $f-P$ can be identified with an element in $\dot{\Lambda}^\alpha (\mathbb{R}^n)$. Moreover, there is 
      a constant $C'  = C'(n,\alpha)$ such that for every $f$ that satisfies  \eqref{eq:sup j finite 23},
          \begin{align*}
      \|f-P\|_{\dot{\Lambda}^\alpha } \leq C' \sup_{j \in \mathbb{Z}} 2^{j\alpha}\|\phi (2^{-j}D)f\|_{L^\infty}.
     \end{align*}
\end{enumerate}
\end{proposition}

\begin{remark}
There is an analogous characterization for the inhomogeneous Lipschitz spaces $\Lambda^\alpha(\mathbb{R}^n)$; see,
e.g., \cite[Theorems~6.3.7]{Grafakos-Modern} and \cite[Theorem 1.3.4 (ii)]{Triebel-2}.
\end{remark}

Proposition \ref{prop:classical L-P homogeneous A} 
imply Littlewood--Paley characterizations of the homogeneous Lipschitz spaces $\dot{\Lambda}^\alpha (\mathbb{R}^n)$ via the spectral decomposition of the
 (positive) Laplacian $-\Delta$. 
We now explain more about this fact. 
It is well known that $-\Delta$ extends to a self-adjoint   operator on
$L^2(\mathbb{R}^n)$, which is also dented by $-\Delta$. By the spectral theory, $-\Delta$ admits
a spectral resolution 
$
-\Delta = \int_0^\infty \lambda dE(\lambda)
$.
Then, for any bounded Borel measurable function $\Phi :[0,\infty) \rightarrow \mathbb{C}$, the operator
\begin{align*}
 \Phi (-\Delta) := \int_0^\infty \Phi(\lambda)dE(\lambda)   
\end{align*}
is a bounded self-adjoint operator on $L^2(\mathbb{R}^n)$.

Now assume that $\Phi \in \mathcal{A}([0,\infty))$ (cf. \eqref{eq:def of A}), and set 
\[
\phi(\xi) := \Phi(|\xi|), \quad \xi \in \mathbb{R}^n,
\]
Then using elementary calculus  
one may check that  $\phi \in \mathcal{S}(\mathbb{R}^n)$. 
By the definition of $\mathcal{F}$ (cf. \eqref{eq:def of Fourier}),
it is easy to see that for any $f \in \mathcal{S}(\mathbb{R}^n)$,
\[
\mathcal{F} ({-\Delta } f) (\xi) = |\xi|^2 \mathcal{F}f(\xi), \quad \xi \in \mathbb{R}^n.
\]
Hence by functional calculus it follows that
\begin{align} \label{eq:Theta sqrt}
\mathcal{F} \big(  \Phi (\sqrt{-\Delta}) f \big) (\xi) = \Phi (|\xi|) \mathcal{F} f(\xi) =\phi(\xi) \mathcal{F}f(\xi), \quad \xi \in \mathbb{R}^n.
\end{align}
Consquently,
\begin{align*}
\Phi (\sqrt{-\Delta}) f (x) = f \ast (\mathcal{F}^{-1} \phi) (x), \quad x \in \mathbb{R}^n.    
\end{align*}
Since $\phi \in \mathcal{S}(\mathbb{R}^n)$, we see that  $\Phi(\sqrt{-\Delta})$ is a continuous operator on $\mathcal{S}(\mathbb{R}^n)$. By duality,
$\Phi(\sqrt{-\Delta})$ extends to  a continuous  operator on $\mathcal{S}'(\mathbb{R}^n)$.
Using this fact and the notation \eqref{eq:notation D}, we may write 
\eqref{eq:Theta sqrt} as
\begin{align} \label{eq:Theta sqrt D}
\Phi(\sqrt{-\Delta}) f= \phi(D)f,  \quad f \in \mathcal{S}'(\mathbb{R}^n).   
\end{align}

Therefore, Proposition \ref{prop:classical L-P homogeneous A} implies the following characterization 
of $\dot{\Lambda}^\alpha (\mathbb{R}^n)$ in terms of Littlewood-Paley-type decompositions constructed via the spectral measures of $-\Delta$. 

\begin{proposition}  \label{prop: Littlewood-Paley homogeneous B} Let $\alpha >0$ and let $\Phi \in C_0^\infty([0,\infty))$ such that
\begin{align*}
 \supp {\Phi} \subset \{\lambda \in \mathbb{R}: 1/2  \leq |\lambda | \leq 2\} \ \  &\mbox{and} \ \  |{\Phi}  (\lambda)| \geq c >0 \ \text{for} \ 3/5  \leq |\lambda| \leq 5/3. 
 \end{align*}
\begin{enumerate}[\rm (i)]
    \item 
There is a constant $C =C(n,\alpha)$ such that for every $f \in \dot{\Lambda}^\alpha (\mathbb{R}^n)$,  
     \begin{align*}
      \sup_{j \in \mathbb{Z}} 2^{j\alpha}\|\Phi(2^{-j}\sqrt{-\Delta})f\|_{L^\infty}
      \leq C \|f\|_{\dot{\Lambda}^\alpha }.
     \end{align*}
\item  For every tempered distribution $f$ that satisfies
    \begin{align} \label{eq:sup j finite 2}
       \sup_{j \in \mathbb{Z}} 2^{j\alpha}\|\Phi(2^{-j}\sqrt{-\Delta})f\|_{L^\infty}
      <\infty,
      \end{align}
     there exists a polynomial $P$ such that $f-P$ can be identified with an element in $\dot{\Lambda}^\alpha (\mathbb{R}^n)$. Moreover, there is 
      a constant $C ' = C'(n,\alpha)$ such that for every $f$ that satisfies  \eqref{eq:sup j finite 2}, we have
          \begin{align*}
      \|f-P\|_{\dot{\Lambda}^\alpha} \leq C' \sup_{j \in \mathbb{Z}} 2^{j\alpha}\|\Phi (2^{-j}\sqrt{-\Delta})f\|_{L^\infty}.
     \end{align*}
\end{enumerate}
\end{proposition}

\subsection{Carleson measure characterization of the classical Lipscthiz spaces}
We now discuss the Carleson measure characterization of the Lipschiz space $\dot{\Lambda}^\alpha (\mathbb{R}^n)$.
It is well known that the dual of the Hardy space $H^1(\mathbb{R}^n)$
is $BMO(\mathbb{R})$, which can be characterized by Carleson measures which are defined as follows.
\begin{definition} \label{def:classical Carleson}
The {\it Carleson measure space} $\mathcal{C}$ is defined as the set
of all nonnegative Borel measures $\mu$ on $\mathbb{R}_+^{n+1} :=\mathbb{R}^n \times (0,\infty)$ such that
\begin{align*}
\|\mu\|_{\mathcal{C}} := \sup_{B} \frac{\mu(\widehat{B})}{|B|} <\infty,
\end{align*}
where the supremum is over all balls $B \subset \mathbb{R}^n$, and 
$\widehat{B} := \{(x,t): x \in B, \ 0< t <r_B\}$, with $r_B$ being the radius of $B$.    
\end{definition}
 
The duals of the Hardy spaces $H^p(\mathbb{R}^n)$, for $0< p <1$, are the  
the Campanato spaces, and the Campanato spaces are  equivalent to the
homogeneous Lipschitz spaces $\dot{\Lambda}^\alpha (\mathbb{R}^n)$. For  these facts, we refer to \cite[Theorem 2.7]{TW}, \cite[Chapter 5]{FS} and
\cite{Greenwald}. This means that $\dot{\Lambda}^\alpha (\mathbb{R}^n)$ are variants of $BMO(\mathbb{R}^n)$.
Hence it is natural to expect that $\dot{\Lambda}^\alpha (\mathbb{R}^n)$ can be characterized via suitable variants 
of the Carleson measures.
To our best knowledge, the first results concerning Carleson measure characterization of $\dot{\Lambda}^\alpha(\mathbb{R}^n)$ for all $\alpha >0$ appears 
in the work of Yang and Yuan \cite{YY-Relations}. To describe their result more clearly, 
we introduce the following definition.
 
\begin{definition} \label{def:gene Carleson}
Let $E = \mathbb{R}^n \times \mathbb{Z}$ and let $\alpha >0$. 
The {\it discrete $\alpha$-Carleson measure space} $\mathcal{C}^\alpha_d$
is defined as the set of all nonnegative Borel measures $\mu$ on $E$ such that
\begin{align*}
\|\mu\|_{\mathcal{C}^\alpha_d} := \sup_{Q} \frac{\mu(\widehat{Q})}{|Q|^{1+\frac{2\alpha}{n}}} <\infty,
\end{align*}
where the supremum is over all dyadic cubes $Q \subset \mathbb{R}^n$, and 
where $\widetilde{Q}$ is the ``discrete tent'' over $Q$ defined by 
\[
\widehat{Q} := Q \times \{j \in \mathbb{Z}: j \geq -\log_2 \ell (Q)\}
\]
with $\ell (Q)$ being 
the sidelength of $Q$.
\end{definition}

In  \cite[Theorem 1 (ii)]{YY-Relations}, if we let $p =q =2$, $s=0$, $\tau = \frac{1}{2}+ \frac{\alpha}{n}$ and $\varphi (\xi) =\Phi(|\xi|)$ for a suitable function $\Phi:[0,\infty) \rightarrow \mathbb{R}$, and
recall that $\dot{B}^{\alpha}_{\infty, \infty} (\mathbb{R}^n)$ coincides with $\dot{\Lambda}^\alpha(\mathbb{R}^n)$ modulo polynomials (see e.g. \cite[Remark 6.5.2]{Grafakos-Modern}), then we deduce the following result.
\begin{proposition} \label{prop:Carleson measure classical}
 Let $\alpha >0$ and let $\Phi \in C_0^\infty([0,\infty))$ such that
\begin{align*}
 \supp {\Phi} \subset \{\lambda \in \mathbb{R}: 1/2  \leq |\lambda | \leq 2\} \ \  &\mbox{and} \ \  |{\Phi}  (\lambda)| \geq c >0 \ \text{for} \ 3/5  \leq |\lambda| \leq 5/3. 
 \end{align*}
For each $f \in \mathcal{S}'(\mathbb{R}^n)$, let $d\mu_{f}$ be a measure on $E:= \mathbb{R}^n \times \mathbb{Z}$
defined by  
\begin{align*}
d\mu_f := |\Phi(2^{-j}\sqrt{-\Delta})f|^2 dx \otimes d\sigma(j), 
\end{align*}
where $dx$ is the Lebesgue measure on $\mathbb{R}^n$ and $d\sigma(j)$ is the  counting measure on $\mathbb{Z}$.
 \begin{enumerate}[\rm (i)]
     \item There is a constant $C = C(n,\alpha)$ such that for every $f \in \dot{\Lambda}^\alpha (\mathbb{R}^n)$,
     \begin{align*}
    \|d\mu_f \|_{\mathcal{C}_d^\alpha} \leq C \|f\|_{\dot{\Lambda}^\alpha}^2.    
     \end{align*}
     \item If $f \in \mathcal{S}'(\mathbb{R}^n)$ such that $d\mu_f \in \mathcal{C}_d^\alpha$, then 
     there exists a polynomial $P$ such that $f -P$ can be identified with an element in $\dot{\Lambda}^\alpha(\mathbb{R}^n)$. Moreover, there is a constant $C '= C'(n,\alpha)$ such that
     for every $f \in \mathcal{S}'(\mathbb{R}^n)$ with $d\mu_f \in \mathcal{C}_d^\alpha$, we have
    \begin{align*}
     \|f-P\|_{\dot{\Lambda}^\alpha}^2 \leq C'\|d\mu_f\|_{\mathcal{C}_d^\alpha}.   
    \end{align*}
 \end{enumerate} 
\end{proposition}

\subsection{Properties of the critical radius function}
We will use the following properties of the critical radius function $\rho$.

\begin{lemma} \label{lem:prop of rho}
{\rm (i)}
 There exist constants $0< \theta <1$ and $C >0$ such that for all $x , y \in \mathbb{R}^n$,
    \begin{align*}
      \rho(y) \leq C\rho(x) \left( 1+ \frac{|x-y|}{\rho(x)}\right)^{\theta}.
    \end{align*}
    \begin{enumerate}[\rm (ii)]
        \item There exists a constant $C >0$ such that for all $\lambda >0$ and for all $x,y \in \mathbb{R}^n$ 
 with $y \in B(x,\lambda \rho(x))$,
\begin{align*}
  \rho(y) \leq C (1 + \lambda)^{\theta} \rho(x).  
\end{align*}
\end{enumerate}
\end{lemma}
\begin{proof}
For (i), see \cite[Lemma 1.4 (c)]{ShenAIF}. Assertion (ii) follows easily from  (i).
\end{proof}

\begin{lemma} \label{lem:Vimpliesrho}
The following two conditions are equivalent:
\begin{enumerate}[\rm (i)]
    \item $|\{x \in \mathbb{R}^n : V(x) > 0\}| \neq 0$.

    \vspace{0.15cm}
    \item $\rho(x) < +\infty$ for every $x \in \mathbb{R}^n$. 
\end{enumerate}
\end{lemma}

\begin{proof}
First we show  $\textrm{(i)} \Rightarrow \textrm{(ii)}$. 
Denote $E:=\{x \in \mathbb{R}^n : V(x) > 0\}$. Since $|E| > 0$, 
there exists a ball $B(x_0,r_0) \subset \mathbb{R}^n$ such that $|B(x_0,r_0) \cap E| >0$. 
Using the fact that $V (y)>0$ for all $y \in B(x_0,r_0) \cap E$, we have
\[
\int_{B(x_0,r_0)}V(y)dy  \geq  \int_{B(x_0,r_0) \cap E}V(y)dy >0.
\]
It follows that
\[
\frac{r_0^2}{|B(x_0,r_0)|} \int_{B(x_0,r_0)} V(y)dy  =:c_0 >0.
\]
From \cite[Lemma 1.2]{ShenAIF} we see that for every $R \geq r_0$, 
\begin{align*}
\frac{R^2}{|B(x_0,R)|} \int_{B(x_0,R)} V(y)dy \geq  C \left( \frac{R}{r_0}\right)^\delta\frac{r_0^2}{|B(x_0,r_0)|} \int_{B(x_0,r_0)} V(y)dy \geq c_0 C\left( \frac{R}{r_0}\right)^\delta,
\end{align*}
where $\delta := 2- (n/q) >0$.
Hence for sufficiently large $R$, we have 
\[
\frac{R^2}{|B(x_0,R)|} \int_{B(x_0,R)} V(y)dy >1.
\]
This implies $\rho (x_0) < +\infty$. It then follows from (i) of Lemma \ref{lem:prop of rho}
that $\rho(x) < +\infty$ for all $x \in \mathbb{R}^n$.

Next we show $\textrm{(ii)} \Rightarrow \textrm{(i)}$. Assume, for the sake of a contradiction, 
that $|\{x \in \mathbb{R}^n : V(x) > 0\}| =0$, i.e., $V(x)=0$ for a.e. $x \in \mathbb{R}^n$.
Then for every $x \in \mathbb{R}^n$ and every $r >0$, we have 
\[
\frac{r^2}{|B(x,r)|} \int_{B(x,r)}V(y) dy =0 \leq 1.
\]
This implies that $\rho(x) =+\infty$, which is a contradiction with (ii).
\end{proof}

\subsection{Calder\'{o}n reproducing formulae }

We record the inhomogeneous and homogeneous Calder\'{o}n reproducing formulae adapted to the operator 
$L$, which will play an important role in our approach. 

First we recall the inhomogeneous Calder\'{o}n reproducing formula adapted to $L$. See \cite[Proposition 5.5]{KP}).
\begin{proposition}  \label{prop:prop of Calderon } 
Let $\Phi_0, \Phi \in C^\infty ([0,\infty))$ such that $\supp \Phi_0 \subset [0,2]$, $\supp \Phi \subset [1/2,2]$, and
\begin{align*}
\Phi_0(\lambda) + \sum_{j =1}^\infty \Phi(2^{-j}\lambda) =1, \quad \forall \lambda \in [0, \infty).    
\end{align*}
Then for any $f \in \mathcal{S}_L'(\mathbb{R}^n)$, we have
\begin{align*}
f = \Phi_0(\sqrt{L})f + \sum_{j =1}^\infty \Phi(2^{-j}\sqrt{L})f    
\end{align*}
with convergence in $\mathcal{S}_L'(\mathbb{R}^n)$.
\end{proposition}

To introduce the homogeneous Calder\'{o}n reproducing formula adapted to $L$, we need to work with the 
factor space $\mathcal{S}_L'(\mathbb{R}^n)/\mathcal{P}_L$. For any sequence $\{f_j + \mathcal{P}_L\}_j \subset \mathcal{S}_L'(\mathbb{R}^n)/\mathcal{P}_L$, 
$f_j + \mathcal{P}_L\rightarrow f + \mathcal{P}_L$ if and only if there exist $\{P_j\}_j \subset \mathcal{P}_L$
and $P \in \mathcal{P}_L$
such that 
\[
f_j + P_j \rightarrow f + P \quad \text{in  } \mathcal{S}_L'(\mathbb{R}^n).
\]
In the following, we will usually not explicitly distinguish between $f \in \mathcal{S}_L'(\mathbb{R}^n)$ and
its equivalence class $f + \mathcal{P}_L$, and we will occasionally write $f \in \mathcal{S}_L'(\mathbb{R}^n)/\mathcal{P}_L$.

The following homogeneous Calder\'{o}n reproducing formula adapted to $L$ was established in \cite[Theorem 3.9]{GKKP1}).
\begin{proposition} \label{prop:prop of polynomial} 
Let $\Phi \in C^\infty ([0,\infty))$ such that  $\supp \Phi \subset [1/2,2]$ and
\begin{align} \label{eq: 1=1}
 \sum_{j\in \mathbb{Z}} \Phi(2^{-j}\lambda) =1, \quad \forall \lambda \in (0, \infty).    
\end{align}
Then for any $f \in \mathcal{S}_L'(\mathbb{R}^n)$, we have
\begin{align*}
f = \sum_{j \in \mathbb{Z}}\Phi(2^{-j}\sqrt{L})f    
\end{align*}
with convergence in $\mathcal{S}_L'(\mathbb{R}^n)/\mathcal{P}_L$. More precisely,  
there exist a sequence of generalized polynomials $\{P_k\}_{k =1}^\infty \subset  \mathcal{P}_L$
and $P \in \mathcal{P}_L$ such that 
\begin{align*}
f  +P =\lim_{k \rightarrow \infty} \left(\sum_{j = - k}^\infty \Phi(2^{-j}\sqrt{L})f + P_k \right) \quad \text{in } \; \mathcal{S}_L'(\mathbb{R}^n).   
\end{align*}
\end{proposition}

\begin{remark}
If the condition $\supp \Phi \subset [1/2,2]$ is replaced by $\supp \Phi \subset [1/(2t),2/t]$ for some $t >0$, 
the conclusion of Proposition \ref{prop:prop of polynomial} still holds. Let us explain this fact. Assume that
$\Phi \in C^\infty([0,\infty))$, 
$\supp \Phi \subset [1/(2t),2/t]$ for some  $t>0$, and $\Phi$ satisfies \eqref{eq: 1=1}.
We define
\begin{align*}
\Omega (\lambda):= \begin{cases}
 1,  & \lambda \in [0, 1/t],   \\
 \Phi(\lambda), & \lambda \in [1/t, \infty). 
 \end{cases}
\end{align*}
Then one may check that $\Omega \in C^\infty ([0,\infty))$, $\supp \Omega \subset [0,2/t]$, and 
\begin{align*}
\sum_{j= -m_1}^{m_2} \Phi(2^{-j}\lambda)  = \Omega (2^{-m_2}\lambda) -\Omega(2^{m_1 +1}\lambda), \quad \lambda \in (0,\infty), \ m_1, m_2 \in \mathbb{N}.
\end{align*}
Then, by the same argument as the proof of \cite[Theorem 3.9]{GKKP1}, we see that 
\begin{align*}
f = \sum_{j \in \mathbb{Z}}\Phi(2^{-j}\sqrt{L})f  \quad \text{in } \mathcal{S}_L'(\mathbb{R}^n)/\mathcal{P}_L.  
\end{align*}
\end{remark}

To construct Littlewood-Paley decompositions, we will use the following fact.  For the proof, we refer to
\cite[Lemma 3.6]{Hu}.

\begin{lemma} \label{lem:construc LP}
Suppose $\Phi_0, \Phi \in \mathcal{A}([0,\infty))$ satisfy the following conditions:
\begin{align*}
  |\Phi_0 (\lambda)| >0     \text{ on }  \{0\leq \lambda < 2\varepsilon\}
  \quad \text{and} \quad 
   |\Phi (\lambda)| >0  \text{ on }  \{\varepsilon/2 <\lambda<2\varepsilon\}
\end{align*}
for some $\varepsilon >0$. Then we have:
\begin{enumerate}[\rm (i)]
    \item There exist $\Psi_0, \Psi \in \mathcal{A}([0,\infty))$ such that $\supp \Psi_0
    \subset [0,2\varepsilon]$, $\supp \Psi \subset [\varepsilon/2,2\varepsilon]$, and 
    \[
    \Psi_0(\lambda)\Phi_0(\lambda) + \sum_{j =1}^\infty \Psi(2^{-j}\lambda) \Phi(2^{-j}\lambda) =1,\quad \forall \lambda \in [0,\infty).
    \]
\item There also exixs $\Omega \in \mathcal{A}([0,\infty))$ such that $\supp \Omega \subset [\varepsilon/2, 2\varepsilon]$ such that 
\[
\sum_{j \in \mathbb{Z}} \Omega (2^{-j}\lambda) \Phi(2^{-j}\lambda) =1, \quad \forall \lambda \in (0,\infty).
\]
\end{enumerate}

\end{lemma}

\subsection{Heat semigroup characterization of $\Lambda_{L}^\alpha (\mathbb{R}^n)$}

In this subsection we review the heat semigroup characterization of $\Lambda_{L}^\alpha (\mathbb{R}^n)$ 
established by De Le\'{o}n-Contreras and Torrea \cite{DT}. First, we recall the following definition.

\begin{definition} (\cite{DT})
We say that a measurable function $f$ on $\mathbb{R}^n$ satisfies the \textit{heat size condition for}  $L$ if  
\begin{itemize}
    \item[(i)] $\int_{\mathbb{R}^n} e^{-\frac{|x|^2}{t}} |f(x)|dx <\infty$ for every
$t>0$;

\medskip

\item[(ii)]  $\displaystyle\lim_{t \rightarrow \infty}\partial_t^\ell e^{-tL}f(x) =0$ for every $\ell \in \mathbb{N}_0$ and every $x \in \mathbb{R}^n$.
\end{itemize}
We denote
\begin{align}
\mathcal{H}_L(\mathbb{R}^n):= \big\{f: f \text{ is measurable } \text{and 
satisfies the heat size condition for } L\big\}.
\end{align}
\end{definition}

 Recall that for each $k \in \mathbb{N}$, 
 $\Theta_{(k)}(\cdot)$ is a function on $[0,\infty)$ defined by 
\begin{align} \label{eq:def of Theta 22}
\Theta_{(k)} (\lambda):=  \lambda^{2k}e^{-\lambda^2}, \quad \lambda \in [0,\infty).
\end{align} 
 Then for any $f \in \mathcal{H}_L(\mathbb{R}^n)$, we have
 \begin{align*}
 \partial_t^k e^{-tL}f=(-L)^ke^{-tL}f  =(-1)^k t^{-k}(tL)^k e^{-tL}f = (-1)^k t^{-k} \Theta_{(k)} (\sqrt{tL})f.
 \end{align*}
Hence, by changing variables $t = s^2$, we have  
 \begin{equation} \label{eq:equal}
 \begin{split}
\sup_{t>0} t^{k- ({\alpha}/{2})} \|\partial_t^k e^{-tL} f\|_{L^\infty} & = 
\sup_{t>0} t^{- ({\alpha}/{2})} \| \Theta_{(k)} (\sqrt{tL})f\|_{L^\infty}  \\
&  = \sup_{s>0} s^{- \alpha} \| \Theta_{(k)} (s\sqrt{L})f\|_{L^\infty} .
 \end{split}
 \end{equation}

\medskip
 With the above notation, we can state the 
 heat semigroup characterization 
of $\Lambda_L^\alpha(\mathbb{R}^n)$ obtained in \cite{DT} as follows.

\begin{proposition} \label{prop:heat semigroup 1} {\rm (\cite[Theorem 1.5]{DT})}
Assume that $|\{x\in \mathbb{R}^n: V(x) >0\}| \neq 0$.
Let $0< \alpha < 2 -(n/q)$ and $k:=\lfloor \alpha/2\rfloor +1$.  
\begin{enumerate}[\rm (i)]
    \item If $f \in \Lambda_{L}^\alpha (\mathbb{R}^n)$, then $f \in \mathcal{H}_L(\mathbb{R}^n)$  and
     $ \sup_{t>0} t^{- \alpha} \| \Theta_{(k)} (t\sqrt{L})f\|_{L^\infty}  < \infty$.
    Furthermore,  
    there exists a constant $C>0$  such that for all $f \in \Lambda_L^\alpha (\mathbb{R}^n)$, 
    \[
    \|\rho(\cdot)^{-\alpha}f(\cdot)\|_{L^\infty} + 
   \sup_{t>0} t^{- \alpha} \| \Theta_{(k)} (t\sqrt{L})f\|_{L^\infty} \leq
    C \|f\|_{\Lambda_{L}^\alpha}.
    \]
    \item 
    If $f \in \mathcal{H}_L(\mathbb{R}^n)$ and 
    $\sup_{t>0} t^{- \alpha} \| \Theta_{(k)} (t\sqrt{L})f\|_{L^\infty}  < \infty$, then 
    $f \in \Lambda_L^\alpha (\mathbb{R}^n)$. Furthermore, 
    there exists a constant $C'>0$ such that for all $f \in \mathcal{H}_L(\mathbb{R}^n)$,
    \[
    \|f\|_{\Lambda^\alpha_L} \leq C ' \Big(\|\rho(\cdot)^{-\alpha}f(\cdot) \|_{L^\infty}+ \sup_{t>0} t^{- \alpha} \| \Theta_{(k)} (t\sqrt{L})f\|_{L^\infty} \Big) .
    \]
\end{enumerate}
\end{proposition}

\begin{remark}
Although the condition $|\{x\in \mathbb{R}^n: V(x) >0\}| \neq 0$ is not explicitly 
claimed in \cite[Theorem 1.5]{DT}, it is actually needed. Indeed, to prove that 
\begin{align} \label{eq:imply}
\rho(\cdot)^{-\alpha}f(\cdot) \in L^\infty(\mathbb{R}^n) \Longrightarrow f \in \mathcal{H}_L(\mathbb{R}^n),
\end{align}
the authors in \cite{DT} establish the following estimate (see \cite[Lemma~2.6]{DT})
\begin{align} \label{eq:2.6}
|\partial_t^\ell e^{-tL}f(x)| \leq C_\ell\|\rho(\cdot)^{-\alpha}f(\cdot)\|_{L^\infty}  \frac{\rho(x)^\alpha}{t^\ell}
\left(1 + \frac{\sqrt{t}}{\rho(x)} \right)^{-M},
\end{align}
and then claim that using this estimate one can deduce \eqref{eq:imply} (see \cite[Proposition 2.7]{DT}). 
However, it is worth noting that 
if $|\{x\in \mathbb{R}^n: V(x) >0\}| = 0$, i.e., $V(x) =0$ for a.e. $x \in \mathbb{R}^n$, then obviously
$\rho \equiv + \infty$, and hence \eqref{eq:2.6} is trivial, and from \eqref{eq:2.6} we can not 
deduce the implication \eqref{eq:imply}. 
\end{remark}

The next result shows that 
$\|\rho(\cdot)^{-\alpha}f(\cdot)\|_{L^\infty}$ is controlled by $\sup_{t>0} t^{- \alpha} \| \Theta_{(k)} (t\sqrt{L})f\|_{L^\infty}$, provided that $f \in \mathcal{H}_L(\mathbb{R}^n)$.

\begin{proposition} \label{prop:heat semigroup 2} {\rm (\cite[Proposition 2.9]{DT})}
Let $\alpha >0$, $k:= \lfloor \alpha /2\rfloor +1$ and  $f \in \mathcal{H}_L(\mathbb{R}^n)$.
Then $ \sup_{t>0} t^{- \alpha} \| \Theta_{(k)} (t\sqrt{L})f\|_{L^\infty}  <\infty$ implies $\rho(\cdot)^{-\alpha}f(\cdot) \in L^\infty (\mathbb{R}^n)$; 
furthermore, there exists a constant $C>0$ such that for all  $f \in \mathcal{H}_L(\mathbb{R}^n)$,
\[
\|\rho(\cdot)^{-\alpha}f(\cdot)\|_{L^\infty} \leq C \sup_{t>0} t^{- \alpha} \| \Theta_{(k)} (t\sqrt{L})f\|_{L^\infty} .
\]
\end{proposition}

The following result shows that we may replace the integer $k$ in Propositions \ref{prop:heat semigroup 1} and
\ref{prop:heat semigroup 2} by 
any integer $m \geq k$. 
\begin{proposition} \label{prop:heat semigroup  3}{\rm (\cite[Proposition 2.4]{DT})}
Let $\alpha >0$ and $k:= \lfloor \alpha /2\rfloor +1$. Let $m$ be an arbitrary integer such that $m \geq k$.
Then there exists a positive constant $C$ such that for all  $f \in \mathcal{H}_L(\mathbb{R}^n)$,
 \begin{align*}
C^{-1}  \sup_{t>0} t^{- \alpha} \| \Theta_{(k)} (t\sqrt{L})f\|_{L^\infty}  
\leq  \sup_{t>0} t^{- \alpha} \| \Theta_{(m)} (t\sqrt{L})f\|_{L^\infty} 
\leq C \sup_{t>0} t^{- \alpha} \| \Theta_{(k)} (t\sqrt{L})f\|_{L^\infty} .
\end{align*}
\end{proposition}

\section{Kernel estimates} \label{sect:kernel estimates}

 When restricted to $C_0^\infty(\mathbb{R}^n)$, the operator $L =-\Delta +V$ is formally self-adjoint.
 It is well known that (see, e.g., \cite{Simon}) the closure of $L|_{C_0^\infty(\mathbb{R}^n)}$ is the unique self-adjoint extension 
 of $L|_{C^\infty_0(\mathbb{R}^n)}$; we denote this extension also by the symbol $L$.

Let $p_t(x,y)$ be the heat kernel, i.e., the integral kernel of the semigroup $e^{-tL}$ generalized 
by $L$.  Since the potential $V$ is nonnegative, it follows from the Feynman-Kac formula that 
$p_t (x,y)$ has a Gaussian upper bound:
\begin{align} \label{eq:GUE}
0 \leq p_t (x,y) \leq (4\pi t)^{- \frac{n}{2}}\exp \left(-\frac{|x-y|^2}{4t}\right).    
\end{align}
When $V$ belongs to the reverse H\"{o}lder
class $RH_q$ for some $q > n/2$, the estimate \eqref{eq:GUE} can be improved 
as follows; see \cite[Theorem 2.11]{DZ1} and \cite[Theorem 1]{Kurata}.

\begin{proposition} \label{prop:heat kernel}
For any $N >0$, there exist positive constants $c$ and $C=C_N$ such that for all $t >0$ and $x,y 
\in \mathbb{R}^n$, 
\begin{align*}
0 \leq p_t(x, y) \leq  C t^{-\frac{n}{2}}\exp\left(-\frac{|x-y|^2}{ct}\right)\left(1+\frac{\sqrt{t}}{\rho(x)}+\frac{\sqrt{t}}{\rho(y)}\right)^{-N}.
\end{align*}
\end{proposition}

\medskip

Recall that $\mathcal{A}([0,\infty))$ is defined by \eqref{eq:def of A}, and note that 
for any $\Phi \in \mathcal{A}([0,\infty))$, 
\[
\Phi^{(2\nu +1)} (0) =0 \quad \text{for every } \nu \in \mathbb{N}_0. 
\]
Based on the Gaussian upper bound of the heat kernel $p_t(x,y)$, the following kernel estimates 
can be deduced by smooth functional calculus; see \cite[Lemma 2.3]{CD} and \cite[Theorem 3.4]{KP}.

\begin{proposition} \label{prop:K Phi L}
Let $\Phi \in \mathcal{A}([0,\infty))$. Then for every $t>0$, the operator 
 $\Phi(t \sqrt{L})$ is an integral operator; furthermore, for any $N>0$, there exists a constant $C =C_N>0$
 such that for all $t >0$ and $x, y\in \mathbb{R}^n$,
\begin{align*}
\big|K_{\Phi(t \sqrt{L})}(x, y)\big| \leq C t^{-n}\left(1+\frac{|x-y|}{t}\right)^{-N},
\end{align*}
where $K_{\Phi(t\sqrt{L})}(\cdot, \cdot)$ denotes the integral kernel of the operator $\Phi(t\sqrt{L})$. 
\end{proposition}

If we assume further that $\Phi$ has compact support and vanishes near the origin, then we can improve the estimate of the kernel $K_{\Phi(t\sqrt{L})}(\cdot, \cdot)$ as follows.
\begin{proposition} \label{prop:kernel phi t L}
Let $\Phi \in \mathcal{A}([0, \infty))$ such that $\supp \Phi \subset [0,R]$ for some $R>0$. 
Then for any $N_1,N_2 >0$, there exists a constant $C=C_{N_1,N_2} >0$ such that
\begin{align} \label{eq:kernel es}
\big| K_{\Phi(t\sqrt{L})}(x,y)\big| \leq C t^{-n}\left(1 + \frac{|x-y|}{t} \right)^{-N_1}\left(1+\frac{t}{\rho(x)} +\frac{t}{\rho(y)}\right)^{-N_2}.
\end{align}
\end{proposition}
\begin{proof}
Define a function 
$\Psi (\lambda) : = e^{\lambda^2} \Phi(\lambda)$, $\lambda \in [0,\infty)$. Then $\Phi(t \sqrt{L})=e^{-t^2 L}\Psi(t \sqrt{L}) $, and hence
\begin{align} \label{eq:KXY}
K_{{{\Phi}}(t \sqrt{L})}(x, y)=\int_{\mathbb{R}^n} p_t(x, z) K_{\Psi(t \sqrt{L})}(z, y) d z.
\end{align}
By Proposition \ref{prop:heat kernel},   we have
\begin{align} \label{eq:estim-1}
\left|p_t(x, z)\right| &\lesssim t^{-n} \exp\left(-\frac{|x-z|^2}{c t^2}\right)\left(1+\frac{t}{\rho(x)}+\frac{t}{\rho(z)}\right)^{-2N_2}  \nonumber\\
& \lesssim t^{-n}\left(1+\frac{|x-z|}{t}\right)^{-\left(N_1+n+1\right)}\left(1+\frac{t}{\rho(x)}\right)^{-2N_2}.
\end{align}
Meanwhile, since $\Phi \in \mathcal{A}([0,\infty))$ and $\supp \Phi \subset [0,R]$, we have $\Psi \in \mathcal{A}([0,\infty))$, and hence by Proposition \ref{prop:K Phi L}
\begin{align} \label{eq:estim-2}
\big|K_{\Psi(t \sqrt{L})}(z, y)\big| \lesssim t^{-n}\left(1+\frac{|z-y|}{t}\right)^{-N_1} .
\end{align}
Inserting \eqref{eq:estim-1} and \eqref{eq:estim-2} 
into \eqref{eq:KXY}, and using the elementary inequality 
\begin{align*}
\left(1+\frac{|x-z|}{t}\right)^{-N_1}\left(1+\frac{|z-y|}{t}\right)^{-N_1} \leq\left(1+\frac{|x-y|}{t}\right)^{-N_1},
\end{align*}
we obtain
{\small \begin{align} \label{eq: est Kxy}
 \big| K_{\Phi(t\sqrt{L})}(x,y)\big|   
 & \lesssim t^{-n}\left(1+\frac{|x-y|}{t}\right)^{-N_1}\left(1+\frac{t}{\rho(x)}\right)^{-2N_2}\int_{\mathbb{R}^n} t^{-n}\left(1+\frac{|x-z|}{t}\right)^{-(n+1)} d z \nonumber \\
& \lesssim t^{-n}\left(1+\frac{|x-y|}{t}\right)^{-N_1}\left(1+\frac{t}{\rho(x)}\right)^{-2N_2}.
\end{align}
Since} $\Phi(t\sqrt{L})$ is a self-adjoint operator, \eqref{eq: est Kxy} also implies that
\begin{align} \label{eq:est Kxy 2}
\big| K_{\Phi(t\sqrt{L})}(x,y)\big|  =\big| K_{\Phi(t\sqrt{L})}(y,x)\big|  \lesssim 
 t^{-n}\left(1+\frac{|y-x|}{t}\right)^{-N_1}\left(1+\frac{t}{\rho(y)}\right)^{-2N_2}.
\end{align}
Combining \eqref{eq: est Kxy} and \eqref{eq:est Kxy 2} gives that
\begin{align*}
\big| K_{\Phi(t\sqrt{L})}(x,y)\big|^2 &\lesssim t^{-2n} \left(1+\frac{|x-y|}{t}\right)^{-2N_1}\left(1+\frac{t}{\rho(x)}\right)^{-2N_2}
 \left(1+\frac{t}{\rho(y)}\right)^{-2N_2} \\
 & \leq t^{-2n} \left(1+\frac{|x-y|}{t}\right)^{-2N_1}\left(1+\frac{t}{\rho(x)} +\frac{t}{\rho(y)}\right)^{-2N_2}.
\end{align*}
Consequently, 
\begin{align*}
\big| K_{\Phi(t\sqrt{L})}(x,y)\big| \lesssim
t^{-n} \left(1+\frac{|x-y|}{t}\right)^{-N_1}\left(1+\frac{t}{\rho(x)} +\frac{t}{\rho(y)}\right)^{-N_2},
\end{align*}
as desired.
\end{proof}

Using the kernel estimate given in Proposition \ref{prop:K Phi L}, one may derive the following almost orthogonality estimate. See, e.g., \cite[Lemma 3.8]{BDL}.

\begin{proposition} \label{prop:AOE}
Let $M\in \mathbb{N}$ and $\Phi, \Psi \in \mathcal{A}([0,\infty))$ such that $(\cdot)^{-2M} \Psi (\cdot) \in
\mathcal{A}([0,\infty))$.  
Then for any $N >0$, there exists $C >0$ such that 
for all $t \geq s >0$ and all $x,y \in \mathbb{R}^n$, 
\begin{align*}
\big| K_{\Phi(t\sqrt{L})\Psi(s\sqrt{L})}(x,y) \big| \leq  C \left(\frac{s}{t} \right)^{2M} t^{-n}   \left(1+ \frac{|x-y|}{t} \right)^{-N}. 
\end{align*}
In particular, if  $\Phi, \Psi \in \mathcal{A}([0,\infty))$ and both of them vanish near the origin, then
for any $M \in \mathbb{N}$ and $N >0$,  there exists $C >0$ such that 
for all $t, s >0$ and all $x,y \in \mathbb{R}^n$, 
\begin{align*}
 \big| K_{\Phi(t\sqrt{L})\Psi(s\sqrt{L})}(x,y) \big| \leq  C \left(\frac{s}{t} \wedge \frac{t}{s}\right)^{2M} (s\vee t) \left(1+ \frac{|x-y|}{s\vee t} \right)^{-N}.   
\end{align*}
\end{proposition}

\section{Littlewood-Paley characterization of $\Lambda_L^\alpha (\mathbb{R}^n)$} \label{sect:proof of first theorem}

This section is devoted to the proof of Theorem \ref{thm:main-1}.
 
\subsection*{Proof of (i) of Theorem \ref{thm:main-1} }
Let $f\in \Lambda_L^\alpha (\mathbb{R}^n)$ and $k = \lfloor {\alpha}/{2}\rfloor +1$. 
Then by (i) of Proposition \ref{prop:heat semigroup 1}, we have  $f \in \mathcal{H}_L(\mathbb{R}^n)$ and
\begin{align*}
\sup _{t>0} t^{-\alpha}\|\Theta_{(k)}(t\sqrt{L})f\|_{L^{\infty}} \lesssim \|f\|_{\Lambda_L^\alpha},
\end{align*}
where $\Theta_{(k)}(\cdot)$ is defined by \eqref{eq:def of Theta}.
Hence, it suffices to show that 
\begin{align} \label{eq:desired ii}
\sup_{j \in \mathbb{Z}}2^{j\alpha} \| \Phi(2^{-j}\sqrt{L})f\|_{L^\infty}\lesssim \sup_{t >0} t^{-\alpha} \|\Theta_{(k)}(t\sqrt{L})f\|_{L^{\infty}}.
\end{align}

Since $|\Theta_{(k)} (\lambda)| >0$ on $(1/2,2)$,  by (ii) of Lemma \ref{lem:construc LP} there exists $\Omega \in \mathcal{A}([0,\infty))$ such that $\supp \Omega \subset [1/2,2]$ and
\begin{align*}
\sum_{\ell \in \mathbb{Z}} \Omega(2^{-\ell} \lambda) \Theta_{(k)} (2^{-\ell} \lambda)=1, \quad \forall \lambda \in (0,\infty).
\end{align*}
Using Proposition \ref{prop:prop of polynomial}, it follows that 
\begin{align*} 
f=\sum_{\ell \in \mathbb{Z}} \Omega(2^{-\ell} \sqrt{L})  \Theta_{(k)}(2^{-\ell} \sqrt{L}) f
\quad 
\text{in } \mathcal{S}_L'(\mathbb{R}^n)/\mathcal{P}_L.
\end{align*}
More precisely,  
there exist a sequence of generalized polynomials $\{P_k\}_{k =1}^\infty \subset  \mathcal{P}_L$
and $P \in \mathcal{P}_L$ such that 
\begin{align} \label{eq:hom cal re for 11}
f  +P =\lim_{k \rightarrow \infty} \left(\sum_{\ell = - k}^\infty \Omega(2^{-\ell} \sqrt{L})  \Theta_{(k)}(2^{-\ell} \sqrt{L})f + P_k \right) \quad \text{in } \; \mathcal{S}_L'(\mathbb{R}^n).   
\end{align}

Since $\Phi$ vanishes near the origin, we have $\Phi(2^{-j}\sqrt{L})P =0$ for every $j \in \mathbb{Z}$ and every $P \in \mathcal{P}_L$. 
Indeed, if $P \in \mathcal{P}_L$, then $L^m P =0$  for some $m \in \mathbb{N}_0$. Hence
\begin{align} \label{eq:psi iden}
\Phi(2^{-j}\sqrt{L})P &= 2^{-2jm}(2^{-j}\sqrt{L})^{-2m}\Phi(2^{-j}\sqrt{L})(L^mP)  \nonumber\\
&= 2^{-2jm}\Psi(2^{-j}\sqrt{L}) (L^m P)\\
& =0, \nonumber
\end{align}
where $\Psi$ is a function on $[0,\infty)$ defined by
$\Psi (\lambda):= \lambda^{-2m}\Phi(\lambda)$. The second identity in \eqref{eq:psi iden} is justified by the fact that  $\Psi \in\mathcal{A} ([0,\infty))$.

From \eqref{eq:Schwartz with x fixed}, \eqref{eq:commute}, \eqref{eq:Phi L f}, \eqref{eq:hom cal re for 11} and
 \eqref{eq:psi iden}, it follows that for every $j \in \mathbb{Z}$ and every $x \in \mathbb{R}^n$,
\begin{align} \label{eq:added}
 \Phi(2^{-j}\sqrt{L})f(x) &= \Phi(2^{-j}\sqrt{L})(f+P)(x) \nonumber\\
 & = \big(f+ P, \ K_{\Phi(2^{-j}\sqrt{L})}(x,\cdot)\big) \nonumber \\
 &= \lim_{k \rightarrow \infty}\sum_{\ell = - k}^\infty\left( \Omega(2^{-\ell} \sqrt{L})  \Theta_{(k)}(2^{-\ell} \sqrt{L})f + P_k, \ K_{\Phi(2^{-j}\sqrt{L})}(x,\cdot)\right) \nonumber\\
  &= \lim_{k \rightarrow \infty}\sum_{\ell = - k}^\infty\left( \Omega(2^{-\ell} \sqrt{L})  \Theta_{(k)}(2^{-\ell} \sqrt{L})f, \ K_{\Phi(2^{-j}\sqrt{L})}(x,\cdot)\right) \nonumber \\
    &=  \sum_{\ell \in \mathbb{Z} } \Phi(2^{-j}\sqrt{L}) \Omega(2^{-\ell} \sqrt{L})  \Theta_{(k)}(2^{-\ell} \sqrt{L})f(x) \nonumber \\
        &=  \sum_{\ell \in \mathbb{Z} }  \Theta_{(k)}(2^{-\ell} \sqrt{L})\Phi(2^{-j}\sqrt{L}) \Omega(2^{-\ell} \sqrt{L}) f(x) \nonumber\\
 & = \sum_{\ell \in \mathbb{Z}} \int_{\mathbb{R}^n}K_{\Phi(2^{-j}\sqrt{L})\Omega(2^{-\ell} \sqrt{L} )}(x,y) \Theta_{(k)}(2^{-\ell} \sqrt{L}) f(y)dy.
\end{align}
Consequently,
\begin{align} \label{eq:sum j ll}
&  2^{j\alpha}\Phi(2^{-j}\sqrt{L})f(x) \nonumber \\
 &\quad\quad  = \sum_{\ell \in \mathbb{Z}} 2^{(j -\ell) \alpha} \int_{\mathbb{R}^n}K_{\Phi(2^{-j}\sqrt{L})\Omega(2^{-\ell} \sqrt{L} )}(x,y)  \big[ 2^{\ell \alpha}\Theta_{(k)}(2^{-\ell} \sqrt{L}) f(y) \big] dy.
\end{align}

Since both $\Phi$ and $\Omega$ belong to $\mathcal{A}([0,\infty))$ and vanish near the origin, it follows from 
Proposition \ref{prop:AOE} that for any $M \in \mathbb{N}$, 
there exists a constant $C$ such that for all $j, \ell \in \mathbb{Z}$ and all $x,y \in \mathbb{R}^n$,
\begin{align*}
\big| K_{\Phi(2^{-j}\sqrt{L})\Omega(2^{-\ell} \sqrt{L} )}(x,y)\big|  \lesssim 2^{-2|j -\ell|M} 2^{(j \wedge \ell)n} (1 + 2^{j \wedge \ell} |x-y|)^{-(n+1)}. 
\end{align*}
Inserting this into \eqref{eq:sum j ll}, and choosing $M > \alpha /2$, we obtain
\begin{align*}
2^{j\alpha}\big| \Phi(2^{-j}\sqrt{L})f(x)\big| 
&\lesssim \Big(\sup_{\nu \in \mathbb{Z}} 2^{\nu \alpha}\big\| \Theta_{(k)}(2^{-\nu} \sqrt{L}) f \big\|_{L^\infty} \Big) \sum_{\ell \in \mathbb{Z}}2^{-2|j -\ell|M}  2^{(j-\ell) \alpha }\\
& \quad\quad \times  \int_{\mathbb{R}^n}  2^{(j \wedge \ell)n} (1 + 2^{j \wedge \ell} |x-y|)^{-(n+1)} dy \\
&\lesssim \Big(\sup_{\nu \in \mathbb{Z}} 2^{\nu \alpha}\| \Theta_{(k)}(2^{-\nu} \sqrt{L}) f \|_{L^\infty} \Big)
\sum_{\ell \in \mathbb{Z}}2^{-|j -\ell|(2M-\alpha)}  \\
&\lesssim \sup_{\nu \in \mathbb{Z}} 2^{\nu \alpha}\| \Theta_{(k)}(2^{-\nu} \sqrt{L}) f \|_{L^\infty} \\
& \leq \sup_{t >0}t^{-\alpha}\| \Theta_{(k)}(t \sqrt{L}) f \|_{L^\infty}.
\end{align*}
Since this estimates holds for all $j \in \mathbb{Z}$ and all $x \in \mathbb{R}^n$, \eqref{eq:desired ii} follows. 

\subsection*{Proof of (ii) of Theorem \ref{thm:main-1} }
Let $f \in \mathcal{S}_L'(\mathbb{R}^n)$
 such that $\sup_{j \in \mathbb{Z}}2^{j
 \alpha} \|\Phi(2^{-j}\sqrt{L})f \|_{L^\infty} <\infty$.
Since $\supp \Phi \subset  [1/2,2]$ and $|\Phi(\lambda)|  >0$ for $3/5 \leq \lambda \leq 5/3$, 
there exists $\Psi \in C^\infty([0,\infty))$ such that $\supp \Psi \subset [1/2,2]$ and 
\begin{align*}
\sum_{j \in \mathbb{Z}} \Psi (2^{-j}\lambda)\Phi(2^{-j}\lambda) =1, \quad \forall \lambda \in [0,\infty).
\end{align*}
Hence, by the homogeneous Calder\'{o}n reproducing formula (cf. Proposition \ref{prop:prop of polynomial}),
\begin{align} \label{eq:converges in S P}
f = \sum_{j \in \mathbb{Z}} \Psi(2^{-j}\sqrt{L})\Phi(2^{-j}\sqrt{L})f  \quad \text{in } \mathcal{S}_L'/\mathcal{P}_L.
\end{align}
Before proceeding further, we note that 
each term $\Psi(2^{-j}\sqrt{L})\Phi(2^{-j}\sqrt{L})f$  in the series on the right-hand side of 
\eqref{eq:converges in S P} is not only a distribution in $\mathcal{S}_L'(\mathbb{R}^n)$ but also a pointwise defined function on $\mathbb{R}^n$. This fact is due to \eqref{eq:commute} and \eqref{eq:Phi L f}.
 
We then divide the rest of the proof into three steps.

 \medskip
\noindent
{\bf Step 1.}  We show that the series of functions
\begin{align} \label{eq:series}
\sum_{j \in \mathbb{Z}} \Psi (2^{-j}\sqrt{L})\Phi(2^{-j}\sqrt{L})f 
\end{align}
converges uniformly on any compact subset of $\mathbb{R}^n$, and
there exists a constant $C >0$ such that for all $x \in \mathbb{R}^n$,
\[
|\tilde{f}(x)| \leq C  \Big( \sup_{j \in \mathbb{Z}}2^{j\alpha} \|\Phi(2^{-j}\sqrt{{L}})f\|_{L^\infty}\Big) \rho(x)^\alpha ,
\]
where $\tilde{f}$ is the sum function of the series \eqref{eq:series}.

To prove that  the series \eqref{eq:series}  converges uniformly on any compact subset of $\mathbb{R}^n$,
by the Heine-Borel (finite covering) theorem 
it suffices to show that for every $x_0 \in \mathbb{R}^n$, the series \eqref{eq:series}  converges uniformly on the ball $B(x_0,\rho(x_0))$. 
Now let $x_0$ be an arbitrary point in $\mathbb{R}^n$.
By (ii) of Lemma \ref{lem:prop of rho}, there exist positive constants $A_1$ and $A_2$, 
 such that for all $x \in B(x_0,\rho(x_0))$, we have
\begin{align} \label{eq:equivalence rho}
   A_1\rho(x_0) \leq  \rho(x) \leq A_2 \rho(x_0).
\end{align}
Note that the constants $A_1$ and $A_2$ are independent of $x_0$ and $x$. 
In addition, since $|\{x \in \mathbb{R}^n: V(x) >0\}| \neq 0$, by  
Lemma \ref{lem:Vimpliesrho} we have 
$\rho(x) < +\infty$ for every $x \in \mathbb{R}^n$.

\medskip
For each $j \in \mathbb{Z}$ such that $2^{-j} \leq \rho(x_0)$,  and all $x \in B(x_0,\rho(x_0))$, 
\begin{align*}
|\Psi(2^{-j}\sqrt{L})\Phi(2^{-j}\sqrt{L})f  (x)| &\leq 2^{-j\alpha}\int_{\mathbb{R}^n} \big|K_{\Psi(2^{-j}\sqrt{L})} (x,y)\big|\big| 2^{j\alpha} \Phi(2^{-j}\sqrt{L})f(y)\big|dy \\
& \leq 2^{-j\alpha}\Big( \sup_{\nu \in \mathbb{Z}} 2^{\nu \alpha}\|\Phi(2^{-\nu}\sqrt{L})f\|_{L^\infty} \Big) \\
& \quad\quad\quad \times \int_{\mathbb{R}^n}
2^{jn} (1 +2^j |x-y|)^{-(n+1)} dy  \\
& \lesssim  \Big( \sup_{\nu \in \mathbb{Z}} 2^{\nu \alpha}\|\Phi(2^{-\nu}\sqrt{L})f\|_{L^\infty} \Big)  2^{-j\alpha},
\end{align*}
where for the second inequality we used the kernel estimate in proposition \ref{prop:K Phi L}.
Since  $\sup_{\nu \in \mathbb{Z}} 2^{\nu \alpha}\|\Phi(2^{-\nu}\sqrt{L})f\|_{L^\infty}$ is finite and the constant term series 
\[
\sum_{j \in \mathbb{Z}: \  2^{-j} \leq \rho(x_0)} 2^{-j\alpha}
\]
converges, using Weierstrass M-test it follows that
\[
\sum_{j \in \mathbb{Z}: \ 2^{-j} \leq \rho(x_0)}\Psi (2^{-j}\sqrt{L})\Phi(2^{-j}\sqrt{L})f(x) 
\]
converges uniformly on $B(x_0,\rho(x_0))$. Furthermore, for all $x \in B(x_0,\rho(x_0))$,
\begin{equation} \label{eq:estimate sum 1} 
\begin{split}
  &  \sum_{j \in \mathbb{Z}:\ 2^{-j} \leq \rho(x_0)}
 \big|\Psi (2^{-j}\sqrt{L})\Phi(2^{-j}\sqrt{L})f(x) \big| \\
  & \quad\quad\lesssim \Big( \sup_{\nu \in \mathbb{Z}} 2^{\nu \alpha}\|\Phi(2^{-\nu}\sqrt{L})f\|_{L^\infty} \Big) 
  \sum_{j \in \mathbb{Z}:\ 2^{-j} \leq \rho(x_0)} 2^{-j\alpha}  \\
 &\quad\quad\lesssim \Big( \sup_{\nu \in \mathbb{Z}} 2^{\nu \alpha}\|\Phi(2^{-\nu}\sqrt{L})f\|_{L^\infty} \Big)\rho(x_0)^\alpha  \\
 & \quad\quad \asymp \Big( \sup_{\nu \in \mathbb{Z}} 2^{\nu \alpha}\|\Phi(2^{-\nu}\sqrt{L})f\|_{L^\infty} \Big)\rho(x)^\alpha,
\end{split}
\end{equation}
where the last equivalence is justified by \eqref{eq:equivalence rho}.

\medskip

If $j \in \mathbb{Z}$ such that $2^{-j} > \rho(x_0)$, we let $M:= \lfloor \alpha \rfloor +1$, and write, for all $x \in B(x_0, \rho(x_0))$,
\begin{align*}
&|\Psi(2^{-j}\sqrt{L})\Phi(2^{-j}\sqrt{L})f  (x)| \\
& \quad \leq \int_{\mathbb{R}^n} \big|K_{\Psi(2^{-j}\sqrt{L})}(x,y) \big|\Phi(2^{-j}\sqrt{L})f(y) \big| \big| dy\\
&\quad = 2^{j(M-\alpha)}\rho(x_0)^M\int_{\mathbb{R}^n} \Big(\frac{2^{-j}}{\rho(x_0)} \Big)^M \big|K_{\Psi(2^{-j}\sqrt{L})} (x,y)\big| \big| 2^{j\alpha}\Phi(2^{-j}\sqrt{L})f(y)\big|dy \\
& \quad \leq 2^{j(M-\alpha)}\rho(x_0)^M  \Big( \sup_{\nu \in \mathbb{Z}}2^{\nu\alpha} \|\Phi(2^{-\nu}\sqrt{L})f\|_{L^\infty} \Big)
\int_{\mathbb{R}^n} \Big(\frac{2^{-j}}{\rho(x_0)} \Big)^M \big|K_{\Psi(2^{-j}\sqrt{L})} (x,y)\big|dy.
\end{align*}
By Proposition \ref{prop:kernel phi t L} with $2^{-j}$ in place of $t$, we have
\begin{align*}
&\int_{\mathbb{R}^n} \Big(\frac{2^{-j}}{\rho(x_0)} \Big)^M \big|K_{\Psi(2^{-j}\sqrt{L})} (x,y)\big|dy  \\
& \quad \lesssim \int_{\mathbb{R}^n} \Big(\frac{2^{-j}}{\rho(x_0)} \Big)^M \frac{2^{jn}}{(1+ 2^j |x-y|)^{(n+1)}} 
\Big( \frac{2^{-j}}{\rho(x)}\Big)^{-M}dy  \\
& \quad \asymp \int_{\mathbb{R}^n}   \frac{2^{jn}}{(1+ 2^j |x-y|)^{(n+1)}}  dy  \\
& \quad \lesssim 1,
\end{align*}
where we also used \eqref{eq:equivalence rho}. It follows that
\begin{align*} 
|\Psi(2^{-j}\sqrt{L})\Phi(2^{-j}\sqrt{L})f  (x)| 
\lesssim  \Big( \sup_{\nu \in \mathbb{Z}}2^{\nu\alpha} \|\Phi(2^{-\nu}\sqrt{L})f\|_{L^\infty} \Big)\rho(x_0)^M2^{j(M-\alpha)} .
\end{align*}
Since the constant term series 
\[
\sum_{j \in \mathbb{Z}:\ 2^{-j} > \rho(x_0)} 2^{j(M-\alpha)}
\]
converges,   applying the Weierstrass M-test we see that
\[
\sum_{j \in \mathbb{Z}:\ 2^{-j} > \rho(x_0)}\Psi (2^{-j}\sqrt{L})\Phi(2^{-j}\sqrt{L})f(x) 
\]
converges uniformly on $B(x_0,\rho(x_0))$. Moreover, for all $x \in B(x_0,\rho(x_0))$ we have
\begin{equation} \label{eq:est sum 2}
\begin{split}
& \sum_{j \in \mathbb{Z}:\ 2^{-j} > \rho(x_0)}
 \big|\Psi (2^{-j}\sqrt{L})\Phi(2^{-j}\sqrt{L})f(x) \big|  \\
 &\quad\quad \lesssim \Big( \sup_{\nu \in \mathbb{Z}}2^{\nu\alpha} \|\Phi(2^{-\nu}\sqrt{L})f\|_{L^\infty} \Big)\rho(x_0)^M \sum_{j \in \mathbb{Z}:\ 2^{-j} > \rho(x_0)} 2^{j(M-\alpha)} \\
 &\quad\quad \asymp \Big( \sup_{\nu \in \mathbb{Z}}2^{\nu\alpha} \|\Phi(2^{-\nu}\sqrt{L})f\|_{L^\infty} \Big) \rho(x_0)^M \rho(x_0)^{\alpha -M} \\
 &\quad\quad \asymp \Big( \sup_{\nu \in \mathbb{Z}}2^{\nu\alpha} \|\Phi(2^{-\nu}\sqrt{L})f\|_{L^\infty} \Big) \rho(x)^\alpha,
\end{split}
\end{equation}
where the last equivalence used \eqref{eq:equivalence rho}.

Therefore, we have proved that the series  \eqref{eq:series} converges uniformly on $B(x_0,\rho(x_0))$.
Since $x_0$ is an arbitrary point in $\mathbb{R}^n$, it follows that the series \eqref{eq:series} converges
uniformly on every compact subset of $\mathbb{R}^n$.
We denote by $\tilde{f}$ the sum function of the series \eqref{eq:series}. Then by \eqref{eq:estimate sum 1} and \eqref{eq:est sum 2} we have, for all $x \in \mathbb{R}^n$, 
\begin{align} \label{eq:estimate of tilde f}
|\tilde{f}(x)| \leq \sum_{j \in \mathbb{Z}}
 \big|\Psi (2^{-j}\sqrt{L})\Phi(2^{-j}\sqrt{L})f(x) \big| \lesssim \Big( \sup_{\nu \in \mathbb{Z}}2^{\nu\alpha} \|\Phi(2^{-\nu}\sqrt{L})f\|_{L^\infty} \Big)  \rho(x)^\alpha .
\end{align}
In other words,
\begin{align} \label{eq:tilde f infinity}
\| \rho(\cdot)^{-\alpha}\tilde{f}(\cdot)\|_{L^\infty} \lesssim \sup_{\nu \in \mathbb{Z}}2^{\nu \alpha} \|\Phi(2^{-\nu}\sqrt{L})f\|_{L^\infty}.
\end{align}

Note that, by (i) of Lemma \ref{lem:prop of rho},  $\rho$ is at most of polynomial growth. Hence 
 \eqref{eq:estimate of tilde f} implies that $\tilde{f}$ can naturally be regarded 
as a distribution in $\mathcal{S}_{L}'(\mathbb{R}^n)$.

\medskip
\noindent
{\bf Step 2.}  We  next show that   
 \begin{align} \label{eq:converge in S}
\tilde{f} = \sum_{j \in \mathbb{Z}}\Psi(2^{-j}\sqrt{L})\Phi(2^{-j}\sqrt{L})f \quad \text{in } \mathcal{S}_L'(\mathbb{R}^n),
 \end{align}
i.e., for every $h \in \mathcal{S}_{L}(\mathbb{R}^n)$, 
\begin{align} \label{eq:converges to 0}
\lim_{\substack{N_1 \rightarrow -\infty\\ N_2 \rightarrow +\infty }} \int_{\mathbb{R}^n} \left[\tilde{f}(x)-\sum_{ N_1 \leq j \leq N_2} 
\Psi(2^{-j}\sqrt{L})\Phi(2^{-j}\sqrt{L})f (x)  \right] h(x)dx  =0.
\end{align}

\medskip

Let $\varepsilon$ be an arbitrary positive real number and $h \in \mathcal{S}_{L} (\mathbb{R}^n)$, $h \neq 0$. We write
\begin{align*}
& \int_{\mathbb{R}^n} \left[\tilde{f}(x)-\sum_{ N_1 \leq j \leq N_2} 
\Psi(2^{-j}\sqrt{L})\Phi(2^{-j}\sqrt{L})f (x)  \right] h(x)dx   \\
& \quad = \int_{B(0,R)} \left[\tilde{f}(x)  -\sum_{ N_1 \leq j \leq N_2} 
 \Psi(2^{-j}\sqrt{L})\Phi(2^{-j}\sqrt{L})f(x)   \right]  h(x)dx  \\
& \quad\quad + \int_{\mathbb{R}^n \backslash B(0,R)} \left[\rho(x)^{-\alpha}\tilde{f}(x) -\sum_{ N_1 \leq j \leq N_2} 
\rho(x)^{-\alpha}\Psi(2^{-j}\sqrt{L})\Phi(2^{-j}\sqrt{L})f(x)   \right] \rho(x)^\alpha h(x)dx\\
& \quad =: I_1 +I_2,
\end{align*}
where $R$ is a sufficiently large radius which will be determined later.

By \eqref{eq:estimate of tilde f} we have, for any $N_1,N_2 \in \mathbb{Z}$ ($N_1 <N_2$) and any $x \in \mathbb{R}^n$,
\begin{align*}
&\left| \rho(x)^{-\alpha}   \tilde{f}(x) -\sum_{ N_1 \leq j \leq N_2} 
\rho(x)^{-\alpha}\Psi(2^{-j}\sqrt{L})\Phi(2^{-j}\sqrt{L})f (x) \right| \\
& \quad \leq \rho(x)^{-\alpha}|\tilde{f}(x)| + \rho(x)^{-\alpha}\sum_{j \in \mathbb{Z}} |\Psi(2^{-j}\sqrt{L})\Phi(2^{-j}\sqrt{L})f(x)| \\
& \quad \lesssim  \sup_{\nu \in \mathbb{Z}}2^{\nu\alpha} \|\Phi(2^{-\nu}\sqrt{L})f\|_{L^\infty} .
\end{align*}
On the other hand,
since $0< \rho(0) <+\infty$,
 by (i) of Lemma \ref{lem:prop of rho}, we have 
\begin{align} \label{eq:rhoxrho0}
\rho(x) \leq \rho(0) \Big(1 + \frac{|x|}{\rho(0)} \Big)^{\theta} \lesssim (1 + |x|)^{\theta}, \quad x \in \mathbb{R}^n.
\end{align}
This along with the fact that $h \in \mathcal{S}_L(\mathbb{R}^n)$ implies that 
\[
\rho(x)^\alpha|h(x)| \lesssim (1 + |x|)^{-n-1}.
\]
Therefore, if $R$ is sufficiently large, we have
\begin{align} \label{eq:fix R}
|I_2| \leq C \Big(\sup_{\nu \in \mathbb{Z}}2^{\nu\alpha} \|\Phi(2^{-\nu}\sqrt{L})f\|_{L^\infty} \Big) \int_{\mathbb{R}^n \backslash B(0,R)} (1 + |x|)^{-n-1}dx  < \frac{\varepsilon}{2}.
\end{align}

We now fix an $R$ such that \eqref{eq:fix R} holds, and then estimate the term $I_1$.  Since  the series \eqref{eq:series} converges uniformly on any compact subset of 
$\mathbb{R}^n$ (by Step 1), there exist $N_1^0, N_2^0 \in \mathbb{Z}$, with $N_1^0 < N_2^0$, such that
for any $N_1 < N_1^0$ and $N_2 > N_2^0$, and for all $x \in B(0,R)$, 
\begin{align*}
 \left|    \tilde{f}(x)  -\sum_{ N_1 \leq j \leq N_2} 
\Psi(2^{-j}\sqrt{L})\Phi(2^{-j}\sqrt{L})f(x) \right|   < \frac{\varepsilon}{2|B(0,R)|\|h\|_{L^\infty}}.
\end{align*}
It follows that for any $N_1 < N_1^0$ and $N_2 > N_2^0$
\begin{align} \label{eq:I-1}
   |I_1|  \leq |B(0,R)| \|h\|_{L^\infty} \frac{\varepsilon}{2|B(0,R)|\|h\|_{L^\infty}} =\frac{\varepsilon}{2}.
\end{align}

Combining \eqref{eq:I-1} and \eqref{eq:fix R}, we see that for any $N_1 < N_1^0$ and $N_2 > N_2^0$,
\begin{align*}
 \int_{\mathbb{R}^n} \left[\tilde{f}(x)-\sum_{ N_1 \leq j \leq N_2} 
\Psi(2^{-j}\sqrt{L})\Phi(2^{-j}\sqrt{L})f (x)  \right] h(x)dx < \frac{\varepsilon}{2} + \frac{\varepsilon}{2} =\varepsilon.
\end{align*}
Hence  \eqref{eq:converges to 0} is true.

\medskip 
\noindent 
{\bf Step 3.} We show that  there exists a generalized polynomial
$P \in \mathcal{P}_L$ such that $f -P =\tilde{f}$ and $\tilde{f} \in \Lambda_L^\alpha (\mathbb{R}^n)$ with 
\begin{align} \label{eq:tilde f lipschitz}
\|\tilde{f}\|_{\Lambda_L^\alpha} \lesssim \sup_{j \in \mathbb{Z}} 2^{j \alpha} \|\Phi(2^{-j}\sqrt{L})f\|_{L^\infty}.
\end{align}

On the one hand, the series \eqref{eq:series} converges in $\mathcal{S}_L'(\mathbb{R}^n)/\mathcal{P}_L$ to $f$ (see \eqref{eq:converges in S P}). On the other hand, the series \eqref{eq:series} converges 
in $\mathcal{S}_L'(\mathbb{R}^n)$ to $\tilde{f}$ (see \eqref{eq:converge in S}), and hence also
 converges 
in $\mathcal{S}_L'(\mathbb{R}^n)/\mathcal{P}_L$ to $\tilde{f}$. Here, we do not explicitly distinguish between 
$g \in \mathcal{S}_L'(\mathbb{R}^n)$ and
its equivalence class $g + \mathcal{P}_L$, as noted earlier on p. 7.
Since $\mathcal{S}_L'(\mathbb{R}^n)/\mathcal{P}_L$ can be identified with the dual of a Frech\'{e}t space (see \cite[Proposition 3.7]{GKKP1}), it is a Hausdorff space and hence any convergent sequence in $\mathcal{S}_L'(\mathbb{R}^n)/\mathcal{P}_L$ must have a unique limit. 
It follows that $f =\tilde{f}$
as elements in $\mathcal{S}_L'(\mathbb{R}^n)/\mathcal{P}$; in other words, there 
exists a generalized polynomial $P \in \mathcal{P}_L$ such that 
\[
f - P =\tilde{f}.
\]
It remains to show \eqref{eq:tilde f lipschitz}.

Recall that $\rho(\cdot)^{-\alpha} \tilde{f}(\cdot) \in L^\infty (\mathbb{R}^n)$ (cf. \eqref{eq:tilde f infinity}). This implies that
 $\tilde{f} \in \mathcal{H}_L(\mathbb{R}^n)$ by \cite[Proposition~2.7]{DT}. Hence it follows from (ii) of Proposition \ref{prop:heat semigroup 1} that
\begin{align*}
\|\tilde{f}\|_{\Lambda^\alpha_L} \lesssim \|\rho(\cdot)^{-\alpha}\tilde{f}(\cdot) \|_{L^\infty}+ \sup_{t>0} t^{- \alpha} \| \Theta_{(k)} (t\sqrt{L})\tilde{f}\|_{L^\infty} .   
\end{align*}
Invoking \eqref{eq:tilde f infinity}, we further see that
\begin{align*}
\|\tilde{f}\|_{\Lambda^\alpha_L} \lesssim \sup_{\nu \in \mathbb{Z}}2^{\nu \alpha} \|\Phi(2^{-\nu}\sqrt{L})f\|_{L^\infty} + \sup_{t>0} t^{- \alpha} \| \Theta_{(k)} (t\sqrt{L})\tilde{f}\|_{L^\infty} .    
\end{align*}
Therefore, in order to show
\eqref{eq:tilde f lipschitz}, it suffices to prove that
\begin{align} \label{eq:Theta controlled by Phi}
\sup_{t>0} t^{- \alpha} \| \Theta_{(k)} (t\sqrt{L})\tilde{f}\|_{L^\infty} \lesssim      \sup_{\nu \in \mathbb{Z}} 2^{\nu\alpha} \|\Phi(2^{-\nu}\sqrt{L})f\|_{L^\infty}.
\end{align}

To show \eqref{eq:Theta controlled by Phi}, we start with the identity 
\eqref{eq:converge in S}. From this identity it follows that
for each $t >0$ and $x \in \mathbb{R}^n$,
\begin{align} \label{eq:pwr}
\Theta_{(k)} (t\sqrt{L})\tilde{f}(x) =  \sum_{j \in \mathbb{Z}}\Theta_{(k)} (t\sqrt{L})\Psi(2^{-j}\sqrt{L})\Phi(2^{-j}\sqrt{L})f (x).
\end{align}
Fix an arbitrary $t \in [1,2]$, and let $\ell_0$ be the (unique) integer such that $t \in [2^{-\ell_0}, 2^{-\ell_0 +1})$. Then 
it follows that
\begin{align} \label{eq:poinwise2}
t^{-\alpha} \big|\Theta_{(k)} (t\sqrt{L})\tilde{f}(x) \big| &\sim 2^{\ell_0 \alpha }  \big|\Theta_{(k)} (t\sqrt{L})\tilde{f}(x)\big| \nonumber \\
& \leq \sum_{j \in \mathbb{Z}}2^{(\ell_0 -j)\alpha} \int_{\mathbb{R}^n} \big| K_{ \Theta_{(k)} (t\sqrt{L})\Psi(2^{-j}\sqrt{L})} (x,y)\big|
2^{j\alpha} | \Phi(2^{-j}\sqrt{L})f (y)|  dy \nonumber \\
& \leq \sum_{j \in \mathbb{Z}}2^{(\ell_0 -j)\alpha}\Big(\sup_{\nu \in \mathbb{Z}} 2^{\nu\alpha} \|\Phi(2^{-\nu}\sqrt{L})f\|_{L^\infty}\Big) \nonumber \\
&\quad\quad \times \int_{\mathbb{R}^n} \big| K_{ \Theta_{(k)} (t\sqrt{L})\Psi(2^{-j}\sqrt{L})} (x,y)\big|  dy.
\end{align}
From the definition of $\Theta_{(k)} (\cdot)$ (cf. \eqref{eq:def of Theta}), we see that $(\cdot)^{-2k}\Theta_{(k)}(\cdot) \in \mathcal{A}([0,\infty))$. Meanwhile, since $\Psi \in C^\infty([0,\infty))$ and $\supp \Psi \subset [1/2,2]$, we also have
$(\cdot)^{-2k}\Psi (\cdot) \in \mathcal{A}([0,\infty))$. 
Hence we may apply Proposition \ref{prop:AOE} to have
\begin{align*}
 \big| K_{ \Theta_{(k)} (t\sqrt{L})\Psi(2^{-j}\sqrt{L})} (x,y)\big| &\lesssim \Big(\frac{t}{2^{-j}} \wedge \frac{2^{-j}}{t} \Big)^{2k}  (t \vee 2^{-j})^{-n} \big[1 +   (t \vee 2^{-j})^{-1} |x-y|\big]^{-(n+1)} \\
 & \sim 2^{-2|j -\ell_0| k}2^{(j \wedge \ell_0)n} (1 + 2^{j \wedge \ell_0} |x-y|)^{-(n+1)}.
\end{align*}
Consequently, 
\begin{align*}
 \int_{\mathbb{R}^n} \big| K_{ \Theta_{(k)} (t\sqrt{L})\Psi(2^{-j}\sqrt{L})} (x,y)\big|  dy 
 \lesssim 2^{-2|j -\ell_0|k}.
\end{align*}
Inserting this into \eqref{eq:poinwise2} we obtain
\begin{align*}
 t^{-\alpha} \big|\Theta_{(k)} (t\sqrt{L})\tilde{f}(x) \big|   
& \lesssim \Big(\sup_{\nu \in \mathbb{Z}} 2^{\nu\alpha} \|\Phi(2^{-\nu}\sqrt{L})f\|_{L^\infty}\Big) \sum_{j \in \mathbb{Z}}2^{(\ell_0 -j)\alpha}2^{-2|j -\ell_0|k} \\
&\leq \Big(\sup_{\nu \in \mathbb{Z}} 2^{\nu\alpha} \|\Phi(2^{-\nu}\sqrt{L})f\|_{L^\infty}\Big) \sum_{j \in \mathbb{Z}}2^{-|j -\ell_0|(2k-\alpha)} \\
& \lesssim \sup_{\nu \in \mathbb{Z}} 2^{\nu\alpha} \|\Phi(2^{-\nu}\sqrt{L})f\|_{L^\infty},
\end{align*}
where we used the fact that $2k -\alpha >0$.
Since this estimate holds for all $t >0$ and $x \in \mathbb{R}^n$, with implicit constant independent of $t$ and $x$,
we see that \eqref{eq:Theta controlled by Phi} is true.

\section{Carleson measure characterization of $\Lambda_L^\alpha (\mathbb{R}^n)$} \label{sect:proof of second theorem}

The main aim of this section is to prove Theorem \ref{thm:main-2}. 
We first present two lemmas containing useful facts.

\begin{lemma} \label{lem:5.2}
Assume that $|\{x \in \mathbb{R}^n: V (x) >0\}| \neq 0$.
If $f \in \mathcal{F}(\mathbb{R}^n)$, then $f$ satisfies the heat size condition for $L$. 
\end{lemma}

\begin{proof}
Since $f \in \mathcal{F}(\mathbb{R}^n)$, there exists $\gamma >0$ such that 
\[
\int_{\mathbb{R}^n} \frac{|f(x)|}{(1 + |x|)^{n+\gamma}} dx <\infty.
\]
Hence, for any $t >0$,
\begin{align*}
\int_{\mathbb{R}^n}\frac{|f(x)|}{ e^{{\frac{|x|^2}{t}} }}dx \lesssim \int_{\mathbb{R}^n}\frac{|f(x)|}{(1 + \frac{|x|}{\sqrt{t}})^{n+\gamma} }dx
&\leq  \int_{\mathbb{R}^n}\frac{|f(x)|}{(1 + |x|)^{n+\gamma}(1 + \sqrt{t} )^{-(n+\gamma)} }dx  \\
& \leq (1 + \sqrt{t})^{n + \gamma}  \int_{\mathbb{R}^n}\frac{|f(x)|}{(1 + |x|)^{n+\gamma} }dx \\
& <\infty.
\end{align*}
It remains to show that for every $\ell \in \mathbb{N}_0$ and $x \in \mathbb{R}^n$, 
$\lim_{t \rightarrow \infty} \partial_t^\ell e^{-tL}f (x) =0$.

To show the last property, we write
\begin{align} \label{eq:tLell}
\partial_t^\ell e^{-tL}f (x) = \int_{\mathbb{R}^n} \partial_t^\ell p_t (x,y)f(y)dy ,
\end{align}
where $p_t (x,y)$ is the heat kernel of $L$. 
By \cite[Lemma 2.1]{DT}, for any $N >0$,  there exist constants $c, C_N >0$ such that
for all $t >0$ and all $x,y \in \mathbb{R}^n$,
\begin{align*}
|\partial_t^\ell p_t (x,y)|& \leq C_N t^{-\ell -\frac{n }{2}} e^{-\frac{|x-y|^2}{ct}} \left(1+  \frac{\sqrt{t}}{\rho(x)} +  \frac{\sqrt{t}}{\rho(y)}\right)^{-N} \\
& \leq C_N 't^{-\ell -\frac{n }{2}}\left( 1 +\frac{|x-y|}{\sqrt{t}}\right)^{-(n +\gamma)}\left( \frac{\sqrt{t}}{\rho(x)}\right)^{-N}. 
\end{align*}
By the triangle inequality, we have, for all $t >0$ and all $x,y \in \mathbb{R}^n$,
\begin{align*}
\left( 1 +\frac{|x-y|}{\sqrt{t}}\right)^{-(n +\gamma)} &\leq \left( 1 +\frac{|y|}{\sqrt{t}}\right)^{-(n +\gamma)}\left( 1 +\frac{|x|}{\sqrt{t}}\right)^{n +\gamma}\\
 & \leq (1 +\sqrt{t})^{n +\gamma} ( 1 + |y| )^{-(n +\gamma)}\left( 1 +\frac{|x|}{\sqrt{t}}\right)^{n +\gamma}.
\end{align*}
It follows that for all $t >0$ and all $x,y \in \mathbb{R}^n$,
\begin{align*}
|\partial_t^\ell p_t (x,y)|\lesssim  t^{-\ell -\frac{n }{2}}(1 +\sqrt{t})^{n +\gamma}  ( 1 + |y| )^{-(n +\gamma)}\left( 1 +\frac{|x|}{\sqrt{t}}\right)^{n +\gamma}\left( \frac{\sqrt{t}}{\rho(x)}\right)^{-N}.
\end{align*}
Inserting this into \eqref{eq:tLell}, we see that for all $t >0$ and all $x \in \mathbb{R}^n$, 
\begin{align*}
   | \partial_t^\ell e^{-tL}f (x) | &\lesssim  t^{-\ell - \frac{n }{2}} (1 +\sqrt{t})^{n +\gamma} \left( \frac{\sqrt{t}}{\rho(x)}\right)^{-N}  \left( 1 + \frac{|x|}{\sqrt{t}}\right)^{n +\gamma}\int_{\mathbb{R}^n} 
      \frac{|f(y)|}{(1 +|y|)^{n+\gamma} }dy \\
     & \lesssim   t^{-\ell - \frac{n }{2}} (1 +\sqrt{t})^{n +\gamma} \left( \frac{\sqrt{t}}{\rho(x)}\right)^{-N}  \left( 1 + \frac{|x|}{\sqrt{t}}\right)^{n +\gamma}.
\end{align*}
We now fix an arbitrary $x \in \mathbb{R}^n$, and choose $N >\max\{0, \gamma -2\ell + 2\}$. 
Then for any $t > \max \{1, |x|^2\}$,
\begin{align} \label{eq:TLE}
 | (tL)^\ell e^{-tL}f (x) | \lesssim t^{-1} \rho(x)^{\gamma -2\ell +2}.
\end{align}
Since $|\{x \in \mathbb{R}^n: V(x) >0\}| \neq 0$, we have $\rho(x) < +\infty$ by Lemma \ref{lem:Vimpliesrho}. 
Hence   \eqref{eq:TLE} implies that
$\lim_{t \rightarrow \infty} \partial_t^\ell e^{-tL}f (x) =0$.
\end{proof}

\begin{lemma} \label{lem: size0}
Assume that $|\{x \in \mathbb{R}^n: V (x) >0\}| \neq 0$.
Let $P \in \mathcal{P}_L(\mathbb{R}^n)$, i.e., $P$ is a generalized polynomial adapted to $L$.
Let further $P \in \mathcal{H}_L(\mathbb{R}^n)$, i.e., $P$ is a measurable function on $\mathbb{R}^n$ satisfying the heat size condition for $L$. Then $P (x)=0$  for a.e. $x \in \mathbb{R}^n$. 
\end{lemma}

\begin{proof}
Since $P \in \mathcal{P}_L(\mathbb{R}^n)$, there exists $m \in \mathbb{N}_0$ such that 
\begin{align} \label{eq:LMP=0}
L^m P =0 \quad \text{in }\mathcal{S}_L'(\mathbb{R}^n).
\end{align}

If $m =0$, then obviously \eqref{eq:LMP=0}  implies that $P(x) =0$ for a.e. $x \in \mathbb{R}^n$.

Now we assume that $m > 0$. On the one hand, \eqref{eq:LMP=0} implies that
\begin{align*} 
\partial_t^m e^{-tL} P =(-L)^me^{-tL}P =(-1)^m e^{-tL} ( L^m P ) =0 \quad 
\text{in } \mathcal{S}_L'(\mathbb{R}^n). 
\end{align*}
On the other hand, the condition $P \in \mathcal{H}_L(\mathbb{R}^n)$ implies that
$\partial_t^\ell e^{-tL}P$ is a pointwise well-defined 
measurable function given by
\[
\partial_t^m e^{-tL} P (x) = \int_{\mathbb{R}^n} \partial_t^m p_t (x,y) P(y)dy, \quad x \in \mathbb{R}^n,
\]
where $p_t (x,y)$ is the heat kernel of $L$. 
Therefore,
\begin{align} \label{eq: tmwtf}
\partial_t^m e^{-tL} P(x) =0 \quad \text{for a.e. } x \in \mathbb{R}^n.
\end{align}

Let $0< \alpha < 2-n/q$ such that $\lfloor \alpha /2 \rfloor +1  \leq m$. Let $k:=\lfloor \alpha /2 \rfloor +1$. By \eqref{eq:equal} and
\eqref{eq: tmwtf},  we have
\[
\sup_{t >0} \|\Theta_{(m)}(t\sqrt{L})P\|_{L^\infty}= \sup_{t >0} t^{m -\frac{\alpha}{2}} \|\partial_t^m e^{-tL}P \|_{L^\infty} =0.
\]
Then using Propositions \ref{prop:heat semigroup 2} and \ref{prop:heat semigroup  3}, it follows that
\begin{align*}
0 \leq \|\rho(\cdot)^{-\alpha}P(\cdot)\|_{L^\infty} \leq C \sup_{t >0} \|\Theta_{(k)}(t\sqrt{L})P \|_{L^\infty} \leq C '\sup_{t >0} \|\Theta_{(m)}(t\sqrt{L})P\|_{L^\infty} =0.
\end{align*}
Consequently, $\|\rho(\cdot)^{-\alpha}P(\cdot)\|_{L^\infty} =0$, i.e.,
\begin{align} \label{eq:rhoalphaP}
\rho(x)^{-\alpha} P(x) =0 \quad \text{for a.e. } x \in \mathbb{R}^n.
\end{align}

Since $|\{x \in \mathbb{R}^n: V (x) >0\}| \neq 0$, by Lemma \ref{lem:Vimpliesrho} we have
$0< \rho(x) < +\infty$ for every $x \in \mathbb{R}^n$. This together with \eqref{eq:rhoalphaP} implies 
that $P(x) =0$ for a.e. $x \in \mathbb{R}^n$.  
\end{proof}

\subsection*{Proof of (i) of Theorem \ref{thm:main-2} } 
 We first show that $\Lambda_L^\alpha (\mathbb{R}^n) \subset \mathcal{F}(\mathbb{R}^n)$.
Note that the condition $|\{x \in \mathbb{R}^n: V(x) >0\}| \neq 0$ implies that
\begin{align*} 
\rho(x) \lesssim (1 + |x|)^{\theta},
\quad x \in \mathbb{R}^n,
\end{align*}
with the implict constant independent of $x$; 
see \eqref{eq:rhoxrho0}. 
Let $f \in \Lambda_L^\alpha (\mathbb{R}^n)$. Then from the definition of $\|\cdot\|_{\Lambda^\alpha_L}$ (see Definition \ref{def:Lipschitz}) we have $|f(x)| \lesssim \rho(x)^{\alpha}$ for a.e. $x \in \mathbb{R}^n$.
It follows that
\begin{align*}
   \int_{\mathbb{R}^n} \frac{|f(x)|}{(1 + |x|)^{n+1 + \theta \alpha}} dx  \lesssim    \int_{\mathbb{R}^n} \frac{\rho(x)^\alpha}{(1 + |x|)^{n+1 + \theta \alpha}} dx  \lesssim \int_{\mathbb{R}^n} \frac{1}{(1 + |x|)^{n+1}} \lesssim 1.
\end{align*}
This shows that $f \in \mathcal{F}(\mathbb{R}^n)$. Hence $\Lambda_L^\alpha (\mathbb{R}^n) \subset \mathcal{F}(\mathbb{R}^n)$.
 
To show the inequality \eqref{eq:1.7-i}, we pick any $f \in \Lambda_L^\alpha (\mathbb{R}^n) \subset \mathcal{F}(\mathbb{R}^n)$. Let $k \in \mathbb{N}$ with $k > \alpha /2$. Recall that $\Theta_{(k)}(\cdot)$ is a function 
on $[0,\infty)$ defined by \eqref{eq:def of Theta}. Then by (i) of Proposition \ref{prop:heat semigroup 1}
and Proposition \ref{prop:heat semigroup  3}, we have
\begin{align*}
\sup_{t >0} t^{-\alpha} \|\Theta_{(k)}(t\sqrt{L})f\|_{L^\infty} \lesssim \|f\|_{\Lambda_L^\alpha}.
\end{align*}
Consequently, for every $t >0$ and every $x \in \mathbb{R}^n$,
\[
|\Theta_{(k)}(t\sqrt{L})f(x)| \lesssim t^\alpha \|f\|_{\Lambda_L^\alpha}.
\]
It follows that 
\begin{align*}
\|d\mu_{(k,f)}\|_{\mathcal{C}^\alpha}   & =\left\| \big|\Theta_{(k)}(t\sqrt{L})f(x)\big|^2\frac{dxdt}{t}\right\|_{\mathcal{C}^\alpha}  \\ 
& =\sup_{\substack{B\subset \mathbb{R}^n \atop B: \text{ball}}}\frac{1}{|B|^{1+\frac{2\alpha}{n}}} \int_0^{r_B} \int_B\big|\Theta_{(k)}(t\sqrt{L})f(x)\big|^2 \frac{dxdt}{t}   \\
 & \lesssim \sup_{\substack{B\subset \mathbb{R}^n \atop B: \text{ball}}}\frac{\|f\|_{\Lambda_L^\alpha}^2}{|B|^{1+\frac{2\alpha}{n}}}  \int_0^{r_B} \int_B t^{2\alpha -1} dxdt   \\
 & \lesssim  \|f\|_{\Lambda_L^\alpha}^2,
\end{align*}
as desired.

\subsection*{Proof of (ii) of Theorem \ref{thm:main-2} }
We divide the proof into several steps.

\medskip
\noindent
{\bf Step 1.} Let $f \in \mathcal{F}(\mathbb{R}^n)$. We show that 
for any $N >0$ and  $\varepsilon >0$, there exists a constant $C >0$ such that for all $\ell \in \mathbb{Z}$
 and $y \in \mathbb{R}^n$,
\begin{align}  \label{eq:int t 12}
&\int_1^2|\Theta_{(k)}(2^{-\ell} t \sqrt{L})f(y)|^2 \frac{dt}{t}\nonumber \\
&\quad\quad \leq C \sum_{j =\ell}^\infty  2^{-(j -\ell) (4k -2\varepsilon)} \int_1^2 \int_{\mathbb{R}^n} \frac{2^{\ell n} |\Theta_{(k)}(2^{-j} t\sqrt{L})f(z)|^2}{(1+2^{\ell}|y-z|)^{2N-n-1}}dz
\frac{dt}{t}.
\end{align}

 Define a function 
 $\Upsilon$ on $[0,\infty)$ by 
 \[
 \Upsilon (\lambda):= e^{-\lambda^2}, \quad \lambda \in [0,\infty).
 \]
Since $|\Upsilon(\lambda)|>0$ for all $\lambda \in [0,\infty)$ and $|\Theta_{(k)}(\lambda)| >0$ for all $\lambda \in (0,\infty)$, 
by (i) of Lemma~\ref{lem:construc LP} there exist $\Psi, \Gamma \in C^\infty ([0,\infty))$ such that $\supp \Psi \subset [1/2,2]$, 
$\supp \Gamma \in [0,2]$, and
\begin{align} \label{eq:inhomo cal rep}
\Gamma(\lambda) \Upsilon(\lambda) + \sum_{j=1}^\infty  \Psi (2^{-j}\lambda)  \Theta_{(k)}(2^{-j} \lambda)=1 , \quad \forall \lambda \in [0,\infty).   
\end{align}

For each $\ell \in \mathbb{Z}$ and $t \in [1,2]$, we may replace $\lambda$ in \eqref{eq:inhomo cal rep} by $2^{-\ell}t\lambda$
to get the identity
\begin{align*}
 \Gamma(2^{-\ell} t\lambda)  \Upsilon(2^{-\ell} t\lambda)+ \sum_{j=1}^\infty \Psi (2^{-(j+\ell)}t\lambda)   \Theta_{(k)}(2^{-(j+ \ell)} t\lambda)=1 , \quad \forall \lambda \in [0,\infty).       
\end{align*}
Then by Proposition \ref{prop:prop of Calderon }, we see that for any $f \in \mathcal{S}_L'(\mathbb{R}^n)$,
any $t \in [1,2]$ and any $\ell \in \mathbb{Z}$,
\begin{align} \label{eq:converges in distribution}
f =  \Gamma(2^{-\ell} t\sqrt{L}) \Upsilon(2^{-\ell}  t\sqrt{L})f + 
\sum_{j=1}^\infty \Psi (2^{-(j+\ell)}t\sqrt{L}) \Theta_{(k)}(2^{-(j+ \ell)} t\sqrt{L}) f,  
\end{align}
with convergence in $\mathcal{S}_L'(\mathbb{R}^n)$. 
Consequently, we have the following pointwise representation: for any $y \in \mathbb{R}^n$, any $t \in [1,2]$ and any $\ell \in \mathbb{Z}$,
\begin{equation} \label{eq:pointwise}
\begin{split}
 \Theta_{(k)}(2^{-\ell} t\sqrt{L})f (y)
& = \Gamma(2^{-\ell} t\sqrt{L}) \Upsilon(2^{-\ell}  t\sqrt{L}) \Theta_{(k)}(2^{-\ell} t\sqrt{L})f (y)\\
&\quad + 
\sum_{j=1}^\infty  \Theta_{(k)}(2^{-\ell} t\sqrt{L}) \Psi (2^{-(j+\ell)}t\sqrt{L}) \Theta_{(k)}(2^{-(j+ \ell)} t\sqrt{L})f (y) .
\end{split}
\end{equation}

Let $N$ be an arbitrary positive number. Since $\Gamma,\Upsilon \in \mathcal{A}([0,\infty))$,
by Proposition~\ref{prop:K Phi L} we have, for every $t \in [1,2]$, $\ell \in \mathbb{Z}$ and $y, z \in \mathbb{R}^n$,
\begin{align*}
 \big| K_{ \Gamma(2^{-\ell} t\sqrt{L})\Upsilon(2^{-\ell} t\sqrt{L})} (y,z)\big| &  \lesssim (2^{-\ell }t)^{-n} 
 \big(1 +(2^{-\ell}t)^{-1}|y-z|\big)^{-N}\\
 & \sim 2^{\ell n} (1 +2^\ell |y-z|)^{-N}.
\end{align*}
Also, since $\Theta_{(k)}, \Psi \in \mathcal{A}([0,\infty))$ with $(\cdot)^{-2k} \Theta_{(k)} (\cdot) \in \mathcal{A}([0,\infty))$, it follows from
 Proposition~\ref{prop:AOE} that for every $t \in [1,2]$, $j \in \mathbb{N}$ and $y, z \in \mathbb{R}^n$,
\begin{align*}
 \big| K_{  \Theta_{(k)}(2^{-\ell} t\sqrt{L}) \Psi (2^{-(j+\ell)}t\sqrt{L})} (y,z)\big|  & \lesssim (2^{-j}t)^{2k }
 (2^{-\ell }t)^{-n} \big(1 + (2^{-\ell}t)^{-1}|y-z| \big)^{-N}\\
 & \sim 2^{-2jk}
 2^{\ell n} (1 + 2^{\ell n} |y-z|)^{-N}.
\end{align*}
Using these kernel estimates, we then have
\begin{align} \label{eq:j=0}
\big|\Gamma(2^{-\ell} t\sqrt{L}) & \Upsilon (2^{-\ell} t\sqrt{L}) \Theta_{(k)}(2^{-\ell} t\sqrt{L})f (y) \big| \nonumber \\
&\leq \int_{\mathbb{R}^n} \big| K_{\Gamma (2^{-\ell} t\sqrt{L}) \Upsilon (2^{-\ell} t\sqrt{L})} (y,z)\big| \big| \Theta_{(k)}(2^{-\ell} t\sqrt{L})f (z)\big|dz \nonumber \\
& \lesssim \int_{\mathbb{R}^n} \frac{2^{\ell n}\big| \Theta_{(k)}(2^{-\ell} t\sqrt{L})f (z)\big|}{(1 + 2^\ell |y-z|)^N}dz,
\end{align}
and for $j =1,2,\dots,$
\begin{align} \label{eq:j geq 1}
&\big| \Theta_{(k)}(2^{-\ell} t\sqrt{L}) \Psi (2^{-(j+\ell)}t\sqrt{L}) \Theta_{(k)}(2^{-(j+ \ell)} t\sqrt{L})f (y) \big| \nonumber\\
&\quad\quad \lesssim   \int_{\mathbb{R}^n} \big| K_{\Theta_{(k)}(2^{-\ell} t\sqrt{L}) \Psi (2^{-(j+\ell)}t\sqrt{L}) } (y,z)\big| \big| \Theta_{(k)}(2^{-(j+ \ell)} t\sqrt{L})f (z)\big|dz \nonumber\\
&\quad\quad \lesssim  2^{-2jk}\int_{\mathbb{R}^n} \frac{2^{\ell n}\big| \Theta_{(k)}(2^{-(j+\ell)} t\sqrt{L})f (z)\big|}{(1 + 2^\ell |y-z|)^N}dz.
\end{align}
Inserting \eqref{eq:j=0} and \eqref{eq:j geq 1} into \eqref{eq:pointwise}, we obtain
\begin{equation} \label{eq:j=ell}
\begin{split}
\big|  \Theta_{(k)}(2^{-\ell} t\sqrt{L})f (y)\big| & \lesssim \int_{\mathbb{R}^n} \frac{2^{\ell n}\big| \Theta_{(k)}(2^{-\ell} t\sqrt{L})f (z)\big|}{(1 + 2^\ell |y-z|)^N}dz \\
& \quad\quad +\sum_{j =1} ^\infty   2^{-2jk}\int_{\mathbb{R}^n} \frac{2^{\ell n}\big| \Theta_{(k)}(2^{-(j+\ell)} t\sqrt{L})f (z)\big|}{(1 + 2^\ell |y-z|)^N}dz  \\
& = \sum_{j =0} ^\infty    2^{-2jk}\int_{\mathbb{R}^n} \frac{2^{\ell n}\big| \Theta_{(k)}(2^{-(j+\ell)} t\sqrt{L})f (z)\big|}{(1 + 2^\ell |y-z|)^N}dz .  
\end{split}
\end{equation}

Applying the Cauchy-Schwarz inequality first for the integral and then for the sum on the right-hand side of \eqref{eq:j=ell}, we have
\begin{align*}
\big|  \Theta_{(k)}(2^{-\ell} t\sqrt{L})f (y)\big| & \lesssim \sum_{j =0} ^\infty   2^{-2jk} \left(\int_{\mathbb{R}^n} \frac{2^{\ell n}\big| \Theta_{(k)}(2^{-(j+\ell)} t\sqrt{L})f (z)\big|^2 }{(1 + 2^\ell |y-z|)^{2N-n-1}}dz \right)^{\frac{1}{2}}\\
& \quad\quad\quad\quad\quad\quad 
\times  \left(\int_{\mathbb{R}^n} \frac{2^{\ell n}}{(1 + 2^\ell |y-z|)^{n+1}}dz \right)^{\frac{1}{2}} \\
& \lesssim \sum_{j =0} ^\infty   2^{-j(2k-\varepsilon)}2^{-j\varepsilon} \left(\int_{\mathbb{R}^n} \frac{2^{\ell n}\big| \Theta_{(k)}(2^{-(j+\ell)} t\sqrt{L})f (z)\big|^2 }{(1 + 2^\ell |y-z|)^{2N-n-1}}dz \right)^{\frac{1}{2}}\\ \\
 &\lesssim \left(\sum_{j =0} ^\infty   2^{-j(4k-2\varepsilon)}\int_{\mathbb{R}^n} \frac{2^{\ell n}\big| \Theta_{(k)}(2^{-(j+\ell)} t\sqrt{L})f (z)\big|^2}{(1 + 2^\ell |y-z|)^{2N -n -1}}dz \right)^{\frac{1}{2}} \\
 & \quad\quad\quad\quad\quad\quad \quad\quad\quad 
\times  \left(\sum_{j =0}^\infty 2^{-2j\varepsilon} \right)^{\frac{1}{2}}\\
 &\lesssim \left(\sum_{j =0} ^\infty   2^{-j(4k-2\varepsilon)}\int_{\mathbb{R}^n} \frac{2^{\ell n}\big| \Theta_{(k)}(2^{-(j+\ell)} t\sqrt{L})f (z)\big|^2}{(1 + 2^\ell |y-z|)^{2N -n -1}}dz \right)^{\frac{1}{2}},
\end{align*}
where $\varepsilon$ is an arbitrarily positive number. 
Consequently, 
\begin{align*}
\big|  \Theta_{(k)}(2^{-\ell} t\sqrt{L})f (y)\big|^2 &\lesssim \sum_{j =0} ^\infty   2^{-j(4k-2\varepsilon)}\int_{\mathbb{R}^n} \frac{2^{\ell n}\big| \Theta_{(k)}(2^{-(j+\ell)} t\sqrt{L})f (z)\big|^2}{(1 + 2^\ell |y-z|)^{2N-n-1}}dz\\
& =  \sum_{j =\ell} ^\infty   2^{-(j-\ell)(4k-2\varepsilon)}\int_{\mathbb{R}^n} \frac{2^{\ell n}\big| \Theta_{(k)}(2^{-j} t\sqrt{L})f (z)\big|^2}{(1 + 2^\ell |y-z|)^{2N-n-1}}dz.
\end{align*}
Since the last estimate holds for all $t \in [1,2]$ and the  implicit constant is independent of $t$, it follows that
\begin{align} 
&\int_1^2|\Theta_{(k)}(2^{-\ell} t \sqrt{L})f(y)|^2 \frac{dt}{t} \nonumber \\
&\quad\quad\quad\quad \lesssim \sum_{j =\ell}^\infty  2^{-(j -\ell) (4k -2\varepsilon)} \int_1^2 \int_{\mathbb{R}^n} \frac{2^{\ell n} |\Theta_{(k)}(2^{-j} t\sqrt{L})f(z)|^2}{(1+2^{\ell}|y-z|)^{2N -n -1}}dz
\frac{dt}{t}.
\end{align}
 
\medskip
\noindent
{\bf Step 2.}
Let $\Phi \in C^\infty ([0,\infty))$ such that $\supp \Phi \subset [1/2,2]$ and $|\Phi(\lambda)| \geq c >0$ for $3/5 \leq \lambda \leq 5/3$, and 
let $f \in \mathcal{F}(\mathbb{R}^n)$. We show that for any $\nu \in \mathbb{Z}$ and $x \in \mathbb{R}^n$, 
\begin{align} \label{eq:int t 122}
2^{2\nu \alpha}\big|\Phi(2^{-\nu}\sqrt{L}) f (x)\big|^2 \leq C\sup_{\ell \in \mathbb{Z}}\sup_{y \in \mathbb{R}^n}2^{2\ell \alpha} \int_1^2 \big|\Theta_{(k)}(2^{-\ell }t\sqrt{L})f(y) \big|^2 \frac{dt}{t}.
\end{align}

 \medskip
Since $\Theta_{(k)}$ is smooth on $[0, \infty)$, decays rapidly at infinity and $\Theta_{(k)} (\lambda) >0$ on $[1/2,2]$,
there exists $\Omega \in C^\infty([0, \infty))$ such that $\supp \Omega \subset [1/2,2]$ and 
\begin{align*}
 \sum_{j \in \mathbb{Z}} \Omega (2^{-j}\lambda) \Theta_{(k)}(2^{-j}\lambda) =1, \quad \forall \lambda \in (0, \infty).   
\end{align*}
Consequently, for every $t \in [1,2]$,
\begin{align*}
 \sum_{j \in \mathbb{Z}} \Omega (2^{-j}t\lambda) \Theta_{(k)}(2^{-j}t\lambda) =1, \quad \forall \lambda \in (0, \infty).   
\end{align*}
Hence by Proposition \ref{prop:prop of polynomial}  we have,  for every $f \in \mathcal{S}_L'(\mathbb{R}^n)$
and every $t \in [1,2]$,
\begin{align} \label{eq:homo cal rep fo}
f =  \sum_{j \in \mathbb{Z}} \Omega (2^{-j}t\sqrt{L})\Theta_{(k)}(2^{-j}t\sqrt{L})f  \quad \text{in } \mathcal{S}_L'(\mathbb{R}^n)/\mathcal{P}_L.
\end{align}
Recall that, since $\Phi$ vanishes near the origin, $\Phi(2^{-\nu}\sqrt{L})P =0$ for every $P 
\in \mathcal{P}_L$ and every $\nu \in \mathbb{Z}$;  see \eqref{eq:psi iden}. Hence, similarly as \eqref{eq:added},
it follows from \eqref{eq:homo cal rep fo} that
for every $\nu \in \mathbb{Z}$, every $t \in [1,2]$ and every $x \in \mathbb{R}^n$,
\begin{equation} \label{eq:Phi nu L}
\begin{split}
\Phi(2^{-\nu}\sqrt{L}) f (x)& =  \sum_{j \in \mathbb{Z}} \Phi(2^{-\nu}\sqrt{L})\Omega (2^{-j}t\sqrt{L})\Theta_{(k)}(2^{-j}t\sqrt{L})f(x)  \\
&=  \sum_{j \in \mathbb{Z}} \int_{\mathbb{R}^n} K_{\Phi(2^{-\nu}\sqrt{L})\Omega (2^{-j}t\sqrt{L})}(x,y) \Theta_{(k)}(2^{-j}t\sqrt{L})f(y)dy.
\end{split}
\end{equation}

Since $\Phi, \Omega \in \mathcal{A}([0,\infty))$ and both of them vanish near the origin, by Proposition \ref{prop:AOE} we have that for any $M \in \mathbb{N}$,
\begin{align*}
&\big| K_{\Phi(2^{-\nu}\sqrt{L})\Omega (2^{-j}t\sqrt{L})}(x,y)\big| \\
&\quad \lesssim  \left(\frac{2^{-\nu}}{2^{-j}t} \wedge \frac{2^{-j}t}{2^{-\nu}}  \right)^{2M}
\big(2^{-\nu} \vee (2^{-j}t)\big)^{-n}\left(1 + \frac{|x-y|}{2^{-\nu} \vee (2^{-j}t)} \right)^{-(n+1)}\\
&\quad \lesssim 2^{-2|j-\nu|M} 2^{(j \wedge \nu)n} (1 + 2^{j\wedge \nu}|x-y|)^{-(n+1)},
\end{align*}
where the implicit constant is independent of $\nu, j \in \mathbb{Z}$ and $t \in [1,2]$.
Inserting this kernel estimate into \eqref{eq:Phi nu L}, we obtain
\begin{align*}
& 2^{\nu \alpha}\big|\Phi(2^{-\nu}\sqrt{L}) f (x)\big| \\
\lesssim  & \sum_{j \in \mathbb{Z}} 
 2^{-2|j-\nu|M}   2^{(\nu -j)\alpha} 
 \int_{\mathbb{R}^n} 2^{(j \wedge \nu)n} (1 + 2^{j\wedge \nu}|x-y|)^{-(n+1)} 2^{j\alpha} \big|\Theta_{(k)}(2^{-j}t\sqrt{L})f(y) \big|dy.
\end{align*}
Using the Cauchy-Schwarz inequality first for the integral and then for the sum on the right-hand side of the last estimate, it follows that for any $\varepsilon >0$,
{\small \begin{align*}
&2^{\nu \alpha}\big|\Phi(2^{-\nu}\sqrt{L}) f (x)\big| \\
\lesssim &  \sum_{j \in \mathbb{Z}} 
 2^{-|j-\nu|(2M-\alpha)}
\left( \int_{\mathbb{R}^n} 2^{(j \wedge \nu)n} (1 + 2^{j\wedge \nu}|x-y|)^{-(n+1)} 2^{2j\alpha} \big|\Theta_{(k)}(2^{-j}t\sqrt{L})f(y) \big|^2dy\right)^{\frac{1}{2}} \\
&\quad\quad\quad\quad\quad\quad\quad \times \left( \int_{\mathbb{R}^n} 2^{(j \wedge \nu)n} (1 + 2^{j\wedge \nu}|x-y|)^{-(n+1)} dy\right)^{\frac{1}{2}} \\
\lesssim &  \sum_{j \in \mathbb{Z}} 
 2^{-|j-\nu|(2M-\alpha -\varepsilon)}2^{-|j-\nu|\varepsilon}  \\
 & \quad\quad\quad\quad \times
\left( \int_{\mathbb{R}^n} 2^{(j \wedge \nu)n} (1 + 2^{j\wedge \nu}|x-y|)^{-(n+1)} 2^{2j\alpha} \big|\Theta_{(k)}(2^{-j}t\sqrt{L})f(y) \big|^2dy\right)^{\frac{1}{2}}  \\
\lesssim & \left( \sum_{j \in \mathbb{Z}} 
 2^{-|j-\nu|(4M-2\alpha -2\varepsilon)}  
\int_{\mathbb{R}^n} 2^{(j \wedge \nu)n} (1 + 2^{j\wedge \nu}|x-y|)^{-(n+1)} 2^{2j\alpha} \big|\Theta_{(k)}(2^{-j}t\sqrt{L})f(y) \big|^2dy\right)^{\frac{1}{2}}\\
 & \quad\quad\quad\quad\quad\quad\quad\quad  \times
\left(\sum_{j \in \mathbb{Z}} 2^{-2|j -\nu|\varepsilon}\right)^{\frac{1}{2}}  \\
\lesssim & \left( \sum_{j \in \mathbb{Z}} 
 2^{-|j-\nu|(4M-2\alpha -2\varepsilon)}  
\int_{\mathbb{R}^n} 2^{(j \wedge \nu)n} (1 + 2^{j\wedge \nu}|x-y|)^{-(n+1)} 2^{2j\alpha} \big|\Theta_{(k)}(2^{-j}t\sqrt{L})f(y) \big|^2dy\right)^{\frac{1}{2}}.
\end{align*}
Consequently,}
\begin{align*}
&2^{2\nu \alpha}\big|\Phi(2^{-\nu}\sqrt{L}) f (x)\big|^2 \\
\lesssim  & \sum_{j \in \mathbb{Z}} 
 2^{-|j-\nu|(4M-2\alpha -2\varepsilon)}  
\int_{\mathbb{R}^n} 2^{(j \wedge \nu)n} (1 + 2^{j\wedge \nu}|x-y|)^{-(n+1)} 2^{2j\alpha} \big|\Theta_{(k)}(2^{-j}t\sqrt{L})f(y) \big|^2dy . 
\end{align*}

Since this estimate holds for all $t \in [1,2]$ and the implicit constant is independent of $t \in [1,2]$, 
applying $\int_1^2 |\cdot| \frac{dt}{t}$ on both sides (noting that the left-hand side is independent of $t$), 
and using Fubini's theorem, we obtain
{\small \begin{align*} 
&2^{2\nu \alpha}\big|\Phi(2^{-\nu}\sqrt{L}) f (x)\big|^2      \\
 \lesssim  & \sum_{j \in \mathbb{Z}} 
 2^{-|j-\nu|(4M-2\alpha -2\varepsilon)}  
 \int_{\mathbb{R}^n} 2^{(j \wedge \nu)n} (1 + 2^{j\wedge \nu}|x-y|)^{-(n+1)} 
\left(2^{2j\alpha} \int_1^2 \big|\Theta_{(k)}(2^{-j}t\sqrt{L})f(y) \big|^2 \frac{dt}{t} \right)dy   \\
 \leq &  \left(\sup_{\ell \in \mathbb{Z}}\sup_{y \in \mathbb{R}^n}2^{2\ell \alpha} \int_1^2 \big|\Theta_{(k)}(2^{-\ell }t\sqrt{L})f(y) \big|^2 \frac{dt}{t} \right)  \sum_{j \in \mathbb{Z}} 
 2^{-|j-\nu|(4M-2\alpha -2\varepsilon)} .
 \end{align*}
We} now fix an $\varepsilon$ and choose $M$ sufficiently large such that $4M -2\alpha -2\varepsilon >0$. Then 
\[
\sum_{j \in \mathbb{Z}} 
 2^{-|j-\nu|(4M-2\alpha -2\varepsilon)} \lesssim 1,
\]
and hence
 \begin{align*}
2^{2\nu \alpha}\big|\Phi(2^{-\nu}\sqrt{L}) f (x)\big|^2  \lesssim \sup_{\ell \in \mathbb{Z}}\sup_{y \in \mathbb{R}^n}2^{2\ell \alpha} \int_1^2 \big|\Theta_{(k)}(2^{-\ell }t\sqrt{L})f(y) \big|^2 \frac{dt}{t}.
\end{align*}
 This proves \eqref{eq:int t 122}.

 \medskip
 \noindent
 {\bf Step 3.} Let $\Phi$ be as in Step 2. We show that there exists a constant $C$ such that for all
 $f \in \mathcal{H}_L(\mathbb{R}^n)$,
\begin{align} \label{eq:controlled by C alpha}
 \sup_{\nu \in \mathbb{Z}}2^{2\nu \alpha} \|\Phi(2^{-\nu}\sqrt{L})f\|_{L^\infty}^2 \leq C   \big\|d\mu_{(k,f)}\big\|_{\mathcal{C}^\alpha} .
\end{align}

\medskip

Since  \eqref{eq:int t 122} holds for all $\nu \in \mathbb{Z}$ and $x \in \mathbb{R}^n$, with the constant $C$
independent of $\nu$ and $x$, we have
\begin{align*}
  \sup_{\nu \in \mathbb{Z}}2^{2\nu \alpha} \|\Phi(2^{-\nu}\sqrt{L})f\|_{L^\infty}^2  \lesssim  
  \sup_{\ell \in \mathbb{Z}}\sup_{y \in \mathbb{R}^n}2^{2\ell \alpha} \int_1^2 \big|\Theta_{(k)}(2^{-\ell }t\sqrt{L})f(y) \big|^2 \frac{dt}{t}.
\end{align*}
This together with \eqref{eq:int t 12} yields
\begin{align*}
&\sup_{\nu \in \mathbb{Z}}2^{\nu \alpha} \|\Phi(2^{-\nu}\sqrt{L})f\|_{L^\infty}  \\
\lesssim &  \sup_{\ell \in \mathbb{Z}} \sup_{y \in \mathbb{R}^n} 2^{2\ell \alpha} \sum_{j =\ell}^\infty  2^{-(j -\ell) (4k -2
\varepsilon)} \int_1^2 \int_{\mathbb{R}^n} \frac{2^{\ell n} \big|\Theta_{(k)}(2^{-j} t\sqrt{L})f(z)\big|^2}{(1+2^{\ell}|y-z|)^{2N -n -1}}dz
\frac{dt}{t} \\
=    &
\sup_{\ell \in \mathbb{Z}} \sup_{y \in \mathbb{R}^n} 2^{2\ell \alpha} \sum_{j =\ell}^\infty  2^{-(j -\ell) (4k -2
\varepsilon)} \int_{2^{-j}}^{2^{-j+1}} \int_{\mathbb{R}^n} \frac{2^{\ell n}\big|\Theta_{(k)}( t\sqrt{L})f(z)\big|^2}{(1+2^{\ell}|y-z|)^{2N-n-1}}dz
\frac{dt}{t}.
\end{align*}

To prove \eqref{eq:controlled by C alpha}, it suffices to show that there exists a constant 
$C >0$ such that for every
$\ell \in \mathbb{Z}$ and evert $y \in \mathbb{R}^n$,
\begin{align} \label{eq:desired}
 \sum_{j =\ell}^\infty  2^{-(j -\ell) (4k -2
\varepsilon)} \int_{2^{-j}}^{2^{-j+1}} \int_{\mathbb{R}^n} \frac{2^{\ell n}\big|\Theta_{(k)}( t\sqrt{L})f(z)\big|^2}{(1+2^{\ell}|y-z|)^{2N-n-1}}dz
\frac{dt}{t} \leq C  2^{-2\ell \alpha}\big\|d\mu_{(k,f)}\big\|_{\mathcal{C}^\alpha},
\end{align}
provided that $\varepsilon < 2k$ and $N \geq n+1$.

Since our  $\alpha$-Carleson measures are defined in terms of balls (see Defintion \ref{def:alpha Carleson continuous}), which are somewhat 
different from the definition given in \cite{YY-Relations} (see also Definition \ref{def:gene Carleson}), 
we have to decompose $\mathbb{R}^n$ into balls instead of dyadic cubes. 
However, we will use balls with radii comparable with $2^{-\ell}$, $\ell \in \mathbb{Z}$, such that 
the balls are finitely overlapped. This means that we will decompose $\mathbb{R}^n$ into balls that are 
essentially dyadic cubes. 
More precisely, for each $\ell \in \mathbb{Z}$, we denote
\[
\mathcal{X}_\ell := 2^{-\ell-k_0}\mathbb{Z}^n,
\]
where $k_0$ is a fixed positive integer such that  $2^{-k_0} < 1/\sqrt{n}$. Then we have the following properties.
\begin{enumerate} 
    \item For each $\ell \in \mathbb{Z}$, 
    \begin{align*}
    \mathbb{R}^n = \bigcup_{x \in \mathcal{X}_\ell}B(x,2^{-\ell + 1}).
    \end{align*}
    \item There exists a fixed integer $m_0$ such that for each $\ell \in \mathbb{Z}$, 
    \begin{align} \label{eq:m0}
     \sum_{x \in \mathcal{X}_\ell} \chi_{B(x,2^{-\ell +1})}(z) \leq m_0  \quad \text{for all } z \in \mathbb{R}^n. 
    \end{align}
    \item There exist a constant $C \geq 1$ such that for all $x \in \mathcal{X}_\ell$ and $y \in \mathbb{R}^n$,
\begin{align} \label{eq:sup inf eq}
 \sup_{z \in B(x,2^{-\ell + 1})} (1 + 2^\ell |y-z|) \leq C \inf_{z \in B(x,2^{-\ell +1})} (1 + 2^{\ell} |y-z|).   
\end{align}
\end{enumerate}

Properties (i) and (ii) are obvious. We show Property (iii). Fix an arbitrary $x \in \mathbb{R}^n$.
We consider the following two cases. 

{\it Case 1}: $y \in B(x, 5 \cdot 2^{-\ell +1})$.
In this case, for every $z \in B(x,2^{-\ell +1})$, we have 
\[
 |y-z| \leq |y-x| + |z -x| \leq 6 \cdot 2^{-\ell +1}.
\]
It follows that $1 \leq 1 + 2^\ell |y-z| \leq 13$ for all $z \in B(x,2^{-\ell +1})$. Hence 
 \eqref{eq:sup inf eq} holds with $C  =13$. 

{\it Case 2}: $y \in \mathbb{R}^n \backslash B(x, 5 \cdot 2^{-\ell +1})$. 
Fny two points $z_1,z_2 \in B(x, 2^{-\ell +1})$, we have
\begin{align} \label{eq:ext}
|y-z_1| \geq |y-z_2| - |z_2 -x| -|z_1 -x| \geq |y-z_2| - 2\cdot 2^{-\ell +1}.
\end{align}
On the other hand,
\[
|y -z_2| \geq |y -x| - |z_2 -x| \geq 5 \cdot 2^{-\ell +1} - 2^{-\ell +1} =4 \cdot 2^{-\ell +1},
\]
i.e., $2^{-\ell +1 } \leq \frac{1}{4} |y-z_2|$.
The latter estimate together with \eqref{eq:ext} yields that 
\begin{align*}
|y-z_1| \geq |y-z_2| - \frac{1}{2} |y-z_2| = \frac{1}{2}|y-z_2|.
\end{align*}
It follows that for any two points   $z_1,z_2 \in B(x, 2^{-\ell +1})$,
\begin{align*}
1 +2^\ell |y -z_1| \geq   \frac{1}{2} (1 +2^\ell |y-z_2|).  
\end{align*}
Hence \eqref{eq:sup inf eq} holds with $C  =2$. 

Therefore, for every $y \in \mathbb{R}^n$,  \eqref{eq:sup inf eq}  
holds with $C =\max\{13,2\} =13$.

\medskip
We are now ready to show \eqref{eq:desired}. Note that for each  $j \geq \ell$, 
{\small \begin{align*}
&  \int_{2^{-j}}^{2^{-j+1}} \int_{\mathbb{R}^n} \frac{2^{\ell n}\big|\Theta_{(k)}( t\sqrt{L})f(z)\big|^2}{(1+2^{\ell}|y-z|)^{2N-n-1}}dz \frac{dt}{t}  \\
& \leq \sum_{x \in \mathcal{X}_\ell}\int_{2^{-j}}^{2^{-j+1}} \int_{B(x,2^{-\ell +1})} \frac{2^{\ell n}\big|\Theta_{(k)}( t\sqrt{L})f(z)\big|^2}{(1+2^{\ell}|y-z|)^{2N-n-1}}dz \frac{dt}{t}  \\
& \leq \sum_{x \in \mathcal{X}_\ell} \left(\sup_{z \in B(x,2^{-\ell +1})} \frac{2^{\ell n }}
{(1+2^{\ell}|y-z|)^{2N-n-1}} \right)  \int_{0}^{2^{-\ell+1}}\int_{B(x,2^{-\ell +1})}  \big|\Theta_{(k)}( t\sqrt{L})f(z)\big|^2 dz \frac{dt}{t}.
\end{align*}
By} \eqref{eq:def of d mu k f} we have
\begin{align*}
 \int_{0}^{2^{-\ell+1}}\int_{B(x,2^{-\ell +1})}  \big|\Theta_{(k)}( t\sqrt{L})f(z)\big|^2 dz \frac{dt}{t} 
& \leq |B(x,2^{-\ell +1})|^{1 +\frac{2\alpha}{n}} \big\| d\mu_{(k,f)}\big\|_{\mathcal{C}^\alpha}\\
 & \sim 2^{-2\ell \alpha} \big\| d\mu_{(k,f)}\big\|_{\mathcal{C}^\alpha} |B(x,2^{-\ell +1})|.
\end{align*}
Using this and \eqref{eq:sup inf eq}, it follows that
\begin{align*}
&  \int_{2^{-j}}^{2^{-j+1}} \int_{\mathbb{R}^n} \frac{2^{\ell n}\big|\Theta_{(k)}( t\sqrt{L})f(z)\big|^2}{(1+2^{\ell}|y-z|)^{2N-n-1}}dz \frac{dt}{t}  \\
& \quad\lesssim  2^{-2\ell \alpha } \big\| d\mu_{(k,f)}\big\|_{\mathcal{C}^\alpha}\sum_{x \in \mathcal{X}_\ell} \left(\inf_{z \in B(x,2^{-\ell +1})} \frac{2^{\ell n }}
{(1+2^{\ell}|y-z|)^{2N-n-1}} \right)   |B(x,2^{-\ell +1})| \\
& \quad \leq  2^{-2\ell \alpha } \big\| d\mu_{(k,f)}\big\|_{\mathcal{C}^\alpha}\sum_{x \in \mathcal{X}_\ell}  \int_{B(x,2^{-\ell +1})} \frac{2^{\ell n}}{(1+2^{\ell}|y-z|)^{2N-n-1}}dz .
\end{align*}
By \eqref{eq:m0}, we have
\begin{align*}
 \sum_{x \in \mathcal{X}_\ell}  \int_{B(x,2^{-\ell +1})} \frac{2^{\ell n}}{(1+2^{\ell}|y-z|)^{2N-n-1}}dz    
 \leq m_0 \int_{\mathbb{R}^n} \frac{2^{\ell n}}{(1+2^{\ell}|y-z|)^{2N-n-1}}dz \lesssim 1,
\end{align*}
provided that $N \geq n+1$. Hence
\begin{align*}
\int_{2^{-j}}^{2^{-j+1}} \int_{\mathbb{R}^n} \frac{2^{\ell n}\big|\Theta_{(k)}( t\sqrt{L})f(z)\big|^2}{(1+2^{\ell}|y-z|)^{2N-n-1}}dz \frac{dt}{t} \lesssim
2^{-2\ell \alpha } \big\| d\mu_{(k,f)}\big\|_{\mathcal{C}^\alpha}.
\end{align*}
Consequently, if $\varepsilon < 2k$, then
\begin{align*}
\text{LHS} \; \eqref{eq:desired}  &\lesssim 2^{-2\ell \alpha } \big\| d\mu_{(k,f)}\big\|_{\mathcal{C}^\alpha} \sum_{j =\ell}^\infty  2^{-(j -\ell) (4k-2\varepsilon)} \\
& \lesssim 2^{-2\ell \alpha } \big\| d\mu_{(k,f)}\big\|_{\mathcal{C}^\alpha},
\end{align*}
as desired.

\medskip
\noindent
{\bf Step 4.} With \eqref{eq:controlled by C alpha} in hand, we aim to complete the proof of (ii) of Theorem \ref{thm:main-2}.

\medskip
Since \eqref{eq:controlled by C alpha} holds, 
by Theorem \ref{thm:main-1} there exists a generalized polynomial $P_f \in \mathcal{P}_L$
such that $f -P_f \in \Lambda^\alpha_L(\mathbb{R}^n)$ and 
\begin{align} \label{eq:f-P_f}
\|f -P_f\|_{ \Lambda^\alpha_L}    \lesssim \big\|d\mu_{(k,f)}\big\|_{\mathcal{C}^\alpha}.
\end{align}
Hence, by (i) of Proposition \ref{prop:heat semigroup 1}, 
$f-P_f \in \mathcal{H}_L(\mathbb{R}^n)$.
On the other hand, since $f \in \mathcal{F}(\mathbb{R}^n)$,  from Lemma \ref{lem:5.2} 
we see that $f \in \mathcal{H}_L(\mathbb{R}^n)$. 
Therefore, $P_f =P_f-f +f \in \mathcal{H}_L(\mathbb{R}^n)$. 
It then follows by Lemma \ref{lem: size0} that $P_f (x) =0$ for a.e. $x \in \mathbb{R}^n$. 
From this and \eqref{eq:f-P_f}, we infer that $f \in  \Lambda^\alpha_L(\mathbb{R}^n)$ and
$\|f \|_{ \Lambda^\alpha_L}    \lesssim \big\|d\mu_{(k,f)}\big\|_{\mathcal{C}^\alpha}$.

 \bigskip
\section*{Acknowledgments}
The authors sincerely thank the anonymous reviewers for the careful reading of the manuscript
and several constructive suggestions, which essentially improved the exposition of this paper.

\end{document}